\documentclass[10pt]{amsart}

\usepackage{amssymb,latexsym,amsmath,amsfonts}
\usepackage[mathscr]{eucal}
\usepackage{mathrsfs}
\usepackage{graphicx}
\usepackage{amscd}
\usepackage{amsthm}
\usepackage{amsxtra}
\usepackage[nointegrals]{wasysym}
\usepackage{bm}
\usepackage{calc}
\usepackage{float}
\usepackage[usenames,dvipsnames]{xcolor}
\usepackage{hyperref}
\usepackage{hyperref}
\hypersetup{colorlinks=true, citecolor=MidnightBlue, linkcolor=BrickRed}

\numberwithin{equation}{section}
\theoremstyle{definition}
\newtheorem{definition}{Definition}[section]
\theoremstyle{definition}
\newtheorem*{ntn}{Notation}
\newtheorem*{rmk}{Remark}
\newtheorem*{rmks}{Remarks}

\newtheorem{remark}[definition]{Remark}

\theoremstyle{plain}
\newtheorem{theorem}[definition]{Theorem}
\newtheorem{lemma}[definition]{Lemma}
\newtheorem{cor}[definition]{Corollary}

\newcommand{\beas}{\begin{eqnarray*}}
\newcommand{\eeas}{\end{eqnarray*}}
\newcommand{\bes} {\begin{equation*}}
\newcommand{\ees} {\end{equation*}}
\newcommand{\be} {\begin{equation}}
\newcommand{\ee} {\end{equation}}
\newcommand{\bea} {\begin{eqnarray}}
\newcommand{\eea} {\end{eqnarray}}

\newcommand{\eps}{\varepsilon}
\newcommand{\zt}{\zeta}
\newcommand{\de} {\delta}
\newcommand{\lam}{\lambda}

\newcommand{\heps}{\hat\eps}

\newcommand\smpartl[2]{\frac{\partial{#1}}{\partial{#2}}}
\newcommand\partl[2]{\dfrac{\partial{#1}}{\partial{#2}}}
\newcommand\smsecpartl[3]{\frac{\partial^2{#1}}{\partial{#2}\text{ }\partial{#3}}}
\newcommand\secpartl[3]{\dfrac{\partial^2{#1}}{\partial{#2}\partial{#3}}}

\newcommand{\bdy}{\partial}
\newcommand{\Om}{\Omega}
\newcommand{\D}{\mathbb{D}}
\newcommand\pkdom[1]{\overline{#1}\times\bdy{#1}}

\newcommand{\cont}{\mathcal{C}}
\newcommand{\bbh}{\mathbb{H}}

\newcommand{\rea}{\operatorname{Re}}
\newcommand{\ima}{\operatorname{Im}}

\newcommand{\K}{\mathcal{K}}

\newcommand{\V}{\mathcal{V}}

\newcommand{\cler}{\mathfrak{l}}
\newcommand{\levip}{\mathfrak{p}}

\newcommand{\bbb}{\mathbb{B}}
\newcommand\conj[1]{\overline{#1}}
\newcommand{\vol}{\operatorname{vol}}
\newcommand{\rvol}{\vol_{_3}\!}
\newcommand{\inte}{\text{int}}

\newcommand\wt[1]{\widetilde{#1}}
\newcommand{\Hess}{\operatorname{Hess}}
\newcommand{\JacC}{\operatorname{J}_\C\!}
\newcommand{\JacR}{\operatorname{J}_\rl\!}
\newcommand{\tran}{^{\operatorname{tr}}}
\newcommand{\ldbr}{\Big<\Big<}
\newcommand{\rdbr}{\Big>\Big>}
\newcommand{\lbr}{\Big<}
\newcommand{\rbr}{\Big>}
\newcommand{\lkor}{l_{\operatorname{kor}}}
\newcommand{\pow}{\operatorname{pow}}
\newcommand{\hpow}{\operatorname{hpow}}
\newcommand{\cell}{\operatorname{cell}}
\newcommand{\hcell}{\operatorname{hcell}}
\newcommand{\ldiv}{\operatorname{ldiv}}
\newcommand{\Div}{\operatorname{div}}
\newcommand{\id}{\operatorname{Id.}}
\newcommand{\dil}{\operatorname{dil}}
\newcommand{\doth}{\cdot_{\scriptscriptstyle\bbh}}
\newcommand\polcl[2]{\mathcal{P}_{#2}{(#1)}}

\newcommand{\CC}{\mathbb{C}^2}

\newcommand{\C} {\mathbb{C}} 

\newcommand{\rl}{\mathbb{R}}

\newcommand{\Rn} {\mathbb{R}^{n}}
\newcommand{\Rd} {\mathbb{R}^{d}}
\newcommand{\pnat} {\mathbb{N}_+} 
\newcommand{\N} {\mathbb{N}}

\begin{document}
\title{Volume Approximations of Strongly Pseudoconvex Domains}
\author{Purvi Gupta}
\address{Department of Mathematics, University of Michigan, Ann Arbor, Michigan, 48105}
\email{purvi.gpt@gmail.com}
\thanks{This work was partially supported by the NSF under grant no. DMS 1161735.}
\begin{abstract}In convex geometry, the Blaschke surface area measure on the boundary of a convex domain can be interpreted in terms of the complexity of approximating polyhedra. This approach is formulated in the holomorphic setting to establish an alternate interpretation of Fefferman's hypersurface measure on boundaries of strongly pseudoconvex domains in $\CC$. In particular, it is shown that Fefferman's measure can be recovered from the Bergman kernel of the domain.
\end{abstract}
\maketitle

%%%%%%%%%%%%%%%%%%%%%%%%%%%%%%%%%%%%%%%%%%%%%%%%%%%%%%
\section{Introduction}\label{sec_intro}
%------------------------
%------BACKGROUND-------
%------------------------
The Fefferman hypersurface measure on the boundary of a $\cont^2$-smooth domain $\Om\subset\C^d$ is the $(2d-1)$-form, $\sigma_\Om$, satisfying 
	\bes
		\sigma_\Om\wedge d\rho=4^{\frac{d}{d+1}}M(\rho)^{\frac{1}{d+1}}\omega_{\C^d},
	\ees
where $\omega_{\C^d}$ is the standard volume form on $\C^d$, $\rho$ is a defining function for $\Om$ with $\Om=\{\rho<0\}$, and 
	\bes
		M(\rho)=	-
			\det
				\begin{pmatrix}
					\rho 						& \rho_{\overline{z_k}} \\
					\rho_{z_j}& \rho_{z_j\overline{z_k}}
			\end{pmatrix}_{1\leq j,k\leq d}.
	\ees
First introduced by Fefferman in \cite{Fe}, this measure is well-defined under the added assumption that $\Om$ is strongly pseudoconvex (defined in Section \ref{sec_Prelim}). Moreover, it does not depend on the choice of $\rho$ and satisfies the following transformation law:
	\bes
		F^*\sigma_{F(\Om)}=|\det \JacC F|^{\frac{2d}{d+1}}\sigma_\Om,
	\ees
where $F$ is a biholomorphism on $\Om$ that is $\cont^2$-smooth on $\conj\Om$.

The Fefferman hypersurface measure shares strong connections with the Blaschke surface area measure (explored in \cite{Bar} and \cite{BaHa}, for instance) studied in affine convex geometry. If $K\subset \Rd$ is a $\cont^2$-smooth convex body, the Blaschke surface area measure on $\bdy K$ is given by
	\bes
		\tilde \sigma_{K}=\kappa^{\frac{1}{d+1}}s_{_{\operatorname{Euc}}},
	\ees
where $\kappa$ and $s_{_{\operatorname{Euc}}}$ are the Gaussian curvature function and the Euclidean surface area form on $\bdy K$, respectively. Its resemblance to the Fefferman measure is reflected in the following identity:
	\bes
		A^*\tilde \sigma_{A(K)}=|\det \JacR A|^{\frac{d-1}{d+1}}\tilde \sigma_K,
	\ees
where $A$ is an affine transformation of $\Rd$. Since its introduction by Blaschke in \cite{Bla}, several mathematicians have extended the notion of affine surface area to arbitrary convex bodies; see \cite{Le} for details. As this measure is invariant under volume-preserving affine maps, it occurs naturally in volume approximations of convex bodies by polyhedra (see \cite[Chap. 1.10]{GrWi} for a survey). The first complete asymptotic result was due to Gruber \cite{Gru} who showed that if $K\subset\Rd$ is a $\cont^2$-smooth strongly convex body, then
	\be\label{eq_grub}
		\inf\{\vol(P\setminus K):P\in\mathcal{P}_n^c\}
			\sim \frac{1}{2} \Div_{d-1}\left(\int_{\bdy K}\tilde\sigma_K\right)^{(d+1)/(d-1)}\frac{1}{n^{2/(d-1)}}
	\ee
as $n\rightarrow\infty$, where $\mathcal{P}_n^c$ is the class of all polyhedra that circumscribe $K$ and have at most $n$ facets, and $\Div_{d-1}$ is a dimensional constant. Ludwig \cite{Lu} later showed that, if the approximating polyhedra are from $\mathcal{P}_n$, the class of {\em all} polyhedra with at most $n$ facets, then
	\be\label{eq_lud}
		\inf\{\vol(K\Delta P):P\in\mathcal{P}_n\}
			\sim \frac{1}{2} \ldiv_{d-1}\left(\int_{\bdy K}\tilde\sigma_K\right)^{(d+1)/(d-1)}\frac{1}{n^{2/(d-1)}}
	\ee
as $n\rightarrow\infty$, where $\Delta$ denotes the symmetric difference between sets and $\ldiv_{d-1}$ is a dimensional constant. In \eqref{eq_grub} and \eqref{eq_lud}, the constants $\Div_{d-1}$ and $\ldiv_{d-1}$ are named after Dirichlet-Voronoi and Laguerre-Dirichlet-Voronoi tilings (see the appendix), respectively, since these are used to prove the formulae. Later, B{\"o}r{\"o}czky \cite{Bo} proved both these formulae for all $\cont^2$-smooth convex bodies. Similar asymptotics have been obtained using other notions of complexity for a polyhedron --- such as the number of vertices.

In \cite{Bar}, Barrett asks whether such relations can be found between the Fefferman hypersurface measure on a pseudoconvex domain and the complexity of approximating analytic polyhedra. An {\em analytic polyhedron} in $\Om$ is a relatively compact subset that is a union of components of any set of the form 
	\bes
		P=\{z\in\Om:|f_j(z)|<1,\ j=1,...,n\}, 
	\ees
where $f_1,...,f_n$ are holomorphic functions in $\Om$. The natural notion of complexity for an analytic polyhedron, $P$, is its order --- i.e., the number of inequalities that define $P$. This setup, however, is not suited for our purpose as demonstrated by a result due to Bishop (Lemma 5.3.8 in \cite{Ho}) which says that any pseudoconvex domain in $\C^d$ can be approximated arbitrarily well (in terms of the volume of the gap) by analytic polyhedra of order at most $2d$. With the help of an example, we indicate where the problem lies. Let $\Om=\D$ be the unit disc in $\C$. Consider the lemniscate-bound domains
	\bes
		P_n:=\left\{z\in\D:\prod_{k=0}^{2n-1}
		\frac{n}{\pi}\left|(z-\exp(\tfrac{k \pi i}{n}))^{-1}\right|<1\right\}.
	\ees
Each $P_n$ has order $1$ and satisfies $\{|z|<1-\pi/n\}\subset P_n\subset\{|z|<1-{\sqrt{3}\pi}/{2n}\}$. Thus, 
	\bes
		\inf\{\vol(\D\setminus P): P\ 
		\text{is an analytic polyhedron of order}\ 1\}=0.
	\ees
If we, instead, declare the complexity of $P_n$ to be $2n$ --- i.e., the number of poles of the function defining $P_n$, then, since $\lim_{n\rightarrow\infty}n\cdot \vol(\D\setminus P_n)\in(0,\infty)$, we can expect results similar to \eqref{eq_grub} and \eqref{eq_lud}.

The above example leads us to a special class of polyhedral objects. For any fixed $f\in\cont(\pkdom{\Om})$, let $\polcl{f}{n}$ be the collection of all relatively compact sets in $\Om$ of the form 
		\bes
			P=\left\{z\in\Om:|f(z,w^j)|>\de_j, j=1,...,n\right\},
		\ees 
where $w^1,...,w^n\in\bdy\Om$ and $ \de_1,...,\de_n>0$. We present a class of functions $f$ for which asymptotic results such as \eqref{eq_grub} and \eqref{eq_lud} can be obtained for domains  in $\CC$.

\begin{theorem}\label{thm_MAIN}
 Let $\Om\subset\subset\CC$ be a $\cont^{4}$-smooth strongly pseudoconvex domain. Suppose $f\in\cont(\conj\Om\times\bdy\Om)$ is such that
	\begin{itemize}	
		\item[$(i)$] $f(z,w)=0$ if and only if $z=w\in\bdy\Om$, and
		\item[$(ii)$] there exist $\eta>1$ and $\tau>0$ such that 
			\be \label{eq_buli}
				f(z,w)=a(z,w)\mathfrak p(z,w)+O\left(\mathfrak p(z,w)^{\eta}\right) 
			\ee
		on $\Om_\tau:=\{(z,w)\in\pkdom{\Om}:||z-w||\leq \tau\}$, where $\mathfrak p$ is the Levi polynomial of some strictly plurisubharmonic defining function of $\Om$ (see Section \ref{sec_Prelim}) and $a$ is some continuous non-vanishing function on $\Om_\tau$.
	\end{itemize}
Then, there exists a constant $\lkor>0$, independent of $\Om$, such that
	\be\label{eq_MAINCONC}
		\inf\{\vol(\Om\setminus P):P\in\polcl{f}{n}\}\sim \frac{1}{2}\lkor \left(\int_{\bdy\Om}
					\sigma_\Om\right)^\frac{3}{2}\frac{1}{\sqrt{n}}
	\ee
as $n\rightarrow \infty$. 
\end{theorem}
In the tradition of $\Div_{d-1}$ and $\ldiv_{d-1}$, the constant $\lkor$ above is named after Laguerre-Kor{\'a}nyi tilings. Any such tiling comes from a collection of Kor{\' an}yi balls $\mathscr{K}$ covering $[0,1]^3$ in $\rl^3=\C\times\rl$ by minimizing the horizontal power functions $\hpow(\cdot,K):\C\times\rl\rightarrow\C\times\rl$ associated to the balls $K$ in $\mathscr{K}$ (see the appendix for more details). If $\hcell(K)$ denotes the tile associated to $K\in\mathscr{K}$, we obtain that
\bes
\lkor=\lim_{n\rightarrow\infty}\sqrt{n}\inf\left\{-\sum\limits_{K\in\mathscr{K}}\int_{\hcell(K)}\hpow(z',K)dz':\#(\mathscr{K})\leq n \right\},
\ees
Such descriptions have been obtained for $\Div_{d-1}$ and $\ldiv_{d-1}$ as well (see \cite{Gru}, \cite{Lu} and \cite{BoLu}).

We believe that our proof of Theorem \ref{thm_MAIN} can be generalized to higher dimensions, although the exposition becomes exceedingly complicated. We, therefore, merely state what we believe to be is the corresponding asymptotic formula when $\Om\subset\subset\C^d$ and $f\in\cont(\overline\Om\times\bdy\Om)$ satisfy the hypothesis of the above theorem: there is a constant $c_d>0$ such that
	\be\label{eq_MAINCONCd}
		\inf\{\vol(\Om\setminus P):P\in\polcl{f}{n}\}\sim \frac{1}{2}c_d \left(\int_{\bdy\Om}
					\sigma_\Om\right)^\frac{d+1}{d}\frac{1}{n^{1/d}}
	\ee
as $n\rightarrow \infty$. Here, $c_d$ is the $d$-dimensional version of $\lkor$. We encourage the reader to compare the exponents and decay rates in \eqref{eq_MAINCONCd}, \eqref{eq_grub} and \eqref{eq_lud}. A common pattern emerges when we realize that the role played by the Euclidean metric on $\rl^{d-1}$ in obtaining \eqref{eq_grub} and \eqref{eq_lud} is played by the Kor{\'a}nyi metric on the $(2d-1)$-dimensional Heisenberg group in the case of \eqref{eq_MAINCONCd}. The former has Hausdorff dimnesion $d'=d-1$, while the latter has Hausdorff dimension $d'=2d$. The exponent of the boundary measure and the power of $1/n$ in all three formulae now have the unified expressions $(d'+2)/d'$ and $2/d'$, respectively.

Let $\operatorname{LP}(\Om)$ denote the class of $f\in\cont(\pkdom{\Om})$ that satisfy conditions $(i)$ and $(ii)$ of Theorem \ref{thm_MAIN}.  Then, $\operatorname{LP}(\Om)$ is invariant under biholomorphisms that extend ($\cont^2$-)smoothly to the boundary. $\operatorname{LP}(\Om)$ is a natural class when working with strongly pseudoconvex domains and contains elements that yield analytic polyhedra. The Henkin-Ramirez generating function (see \cite[\S 3]{Ra} for details) is one such choice of $f$. So are $K_\Om^{-1/(d+1)}$ and $S_\Om^{-1/d}$, where $K_\Om$ and $S_\Om$ denote the Bergman kernel and Szeg{\H o} kernel on $\Om\subset\C^d$, respectively. In fact, these two choices of $f$ are almost analytic extensions of any defining function of $\Om$. Since the Bergman kernel and almost analytic extensions of defining functions make sense in a context larger than that of strongly pseudoconvex domains, these provide potential candidates for $f$ to obtain results like Theorem \ref{thm_MAIN} in a more general setting. We support this fact with an example where the Fefferman hypersurface measure, though not defined everywhere, is zero almost everywhere with respect to the Hausdorff measure on the boundary. Let $\Om=\D^2$ and $f(z,w)=(1-z_1\overline{w_1})(1-z_2\overline{w_2})$. Then, by choosing appropriate $f$-cuts with sources on the distinguished boundary, it can be shown that
	\bes
		\lim_{n\rightarrow \infty}\sqrt{n}
		\inf\{\vol(\Om\setminus P):P\in\polcl{f}{n}\}=0
	\ees
as $n\rightarrow \infty$. Note that $f$ is a scalar multiple of $K_{\D^2}^{-1/2}$. 

%------------------------
%------PAPER PLAN-------
%------------------------
\noindent {\bf Organization of the paper.} Definitions, notation and terminology that feature in multiple sections are collected in Section \ref{sec_Prelim}. The proof of Theorem \ref{thm_MAIN} is spread over subsequent sections. A critical lemma allows us to pass from $\operatorname{LP}(\Om)$ to a single representative --- this lemma and other technical issues are dealt with in Section \ref{sec_techlemm}. In Section \ref{sec_approxmodel}, we address the problem for certain model domains and model polyhedra. The rate of decay and the relevant exponents in \eqref{eq_MAINCONC} become evident in this section. We move from the model to the general case (locally), and from the local to the global case in Sections \ref{sec_localest} and \ref{sec_mainproofs}, respectively. The appendix contains a brief exposition on power diagrams in the Euclidean plane, and introduces a new tiling problem on the Heisenberg group. The latter emerged naturally in the course of this work, and seems indispensable in proving Theorem \ref{thm_MAIN} (in particular, Lemma \ref{lem_adm} from the appendix is a crucial component of Lemma \ref{lem_locfin}). The appendix also contains bounds for $\lkor$.

%------------------------
%------THANKS-------
%------------------------
\noindent {\bf Acknowledgements.} The author is grateful to her adviser, David Barrett, for suggesting this problem to her, and supporting this work with constant encouragement and timely mathematical insights. She would also like to thank Dan Burns for some very useful discussions. Lastly, the author wishes to thanks the referee for his/her detailed comments that have vastly helped improve this paper. 

%%%%%%%%%%%%%%%%%%%%%%%%% 
%%%%%%%%%%%%%%%%%%%%%%%%% 
%%%%%%%%%%%%%%%%%%%%%%%%% Notations, Definitions, Terminology %%%%%%%%%%%%%%%%%%%%%%%%%%%%
%%%%%%%%%%%%%%%%%%%%%%%%% 
%%%%%%%%%%%%%%%%%%%%%%%%% 
%\newpage
\section{Preliminaries}\label{sec_Prelim}
In this article, $\pnat$ denotes the set of all positive natural numbers. For $D\subseteq\Rn$, $\cont(D)$ is the set of all continuous functions on $D$, and $\cont^k(D)$, $k\geq 1$, denotes the set of all functions that are $k$-times continuously differentiable in some open neighborhood of $D$. If $A\subset B \subset\Rn$, $\inte_BA$ is the interior of $A$ in the relative topology of $B$. The tranpsose of a vector $v$ is denoted by $v\tran$. When well defined, $\JacR f(x)$ and $\Hess_\rl f(x)$ denote the real Jacobian and Hessian matrices, respectively, of $f$ at $x$, $\JacC f(z)$ is the complex Jacobian matrix of $f$ at $z$, and $f^*$ denotes the pull-back operator induced by $f$ on differential forms and measures. For brevity, we often abbreviate $\smpartl{f}{x}$ and $\smsecpartl{f}{x}{y}$ to $f_x$ and $f_{xy}$, respectively. In $\CC$, we employ the notation
\begin{itemize}
\item $z=(z_1,z_2)=(x_1+iy_1,x_2+iy_2)$, $w=(w_1,w_2)=(u_1+iv_1,u_2+iv_2)$ for  points;
\item $\bbb_2(z;r)$ for the Euclidean ball centered at $z$ and of radius $r$;
\item $\lbr \cdot,\cdot\rbr$ for the complex pairing between a co-vector and a vector;
\item $``\ '\ "$  to indicate projection onto $\{y_2=0\}=\C\times\rl$;
\item $A	^{\operatorname{res}}$ for $(A\big|_{\{y_2=0\}})':\C\times\rl\rightarrow\C\times\rl$, where $A:\CC\rightarrow\CC$;
\item $\vol$ for the Lebesgue measure in $\CC$;
\item $\rvol$ for the Lebesgue measure in $\C\times\rl$, and
\item $s$ for the standard Euclidean surface area measure on the boundary of a smooth domain.
\end{itemize}

In our analogy between convex and complex analysis, the role of convexity is played by pseudoconvexity:
\begin{definition} A $\cont^2$-smooth domain $\Om\subset\C^d$ is called {\em strongly pseudoconvex} if it admits a defining function $\rho$ in a neighborhood $U\supset\overline\Om$ such that
	\be\label{eq_levicon}
		\sum_{1\leq j,k\leq d}\secpartl{\rho}{z_j}{\conj{z_k}}(z)v_j\conj{v_k}> 0	
		\quad\text{for}\ z\in\bdy\Om\ \text{and}\ v=(v_1,...,v_d)\in\C^d\setminus\{0\}\ \text{satisfying}\ \sum_{j=1}^d\partl{\rho}{z_j}(z)v_j=0.
	\ee 
A (possibly non-smooth) domain $\Om\subset\C^d$ is called {\em pseudoconvex} if it can be exhausted by strongly pseudoconvex domains, i.e, $\Om=\cup_{j\in\rl}\Om_j$ with each $\Om_j$ strongly pseudoconvex and $\Om_{j}\subseteq\Om_k$ for $j<k$.  
\end{definition}

\begin{rmk} We will heavily use the fact that any strongly pseudoconvex domain $\Om$ admits a defining function $\rho$ which is {\em strictly plurisubharmonic} --- i.e., \eqref{eq_levicon} holds for all $z\in U$ and $v\in\C^d\setminus\{0\}$. 
\end{rmk}

We reintroduce the polyhedral objects of our study.	
%------
%--Def I
%------
\begin{definition}\label{def_fpoly}
Let $\Om\subset\CC$ be a domain and $f\in\cont(\pkdom{\Om})$. Given a compact set $J\subset\bdy\Om$, an {\em $f$-polyhedron over $J$} is any set of the form 
	\bes
	 P=\{z\in\Om:|f(z,w^j)|>\de_j, j=1,...,n\},\ \ \  (w^j,\de_j)\in\bdy\Om\times(0,\infty),
	\ees 
such that $J\subset \bdy\Om\setminus\conj P$ and for every $j\in\{1,...,n\}$, $|f(z,w^j)|<\de_j$ for some $z\in J$. If $\Om$ is bounded, then an $f$-polyhedron over $\bdy\Om$ is simply called an {\em $f$-polyhedron}. We call 
\begin{itemize}
	\item  each $(w^j,\de_j)$ a {\em source-size pair} of $P$;
	\item each $C(w^j,\de_j;f):=\{z\in\conj\Om:|f(z,w^j)|\leq \de_j\}$ a
			 {\em cut} of $P$; 
	\item each $F(w^j,\de_j;f):=\{z\in\conj\Om:|f(z,w^j)|=\de_j, |f(z,w^l)|\geq\de_l,\ 
			l\neq j\}$ a {\em facet} of $P$;
	\item  $(w^1,...,w^n)$ and $(\de_1,...,\de_n)$ the {\em source-tuple} and {\em size-tuple} of $P$, respectively.
\end{itemize}
\end{definition}
We emphasize that, by definition, the cuts of an $f$-polyhedron over $J$ cover $J$, and each of its cuts intersects $J$ non-trivially. 
%------
%--Rmks
%------
\begin{rmks} When there is no ambiguity in the choice of $f$, we drop any reference to it from our notation for cuts and facets. Repetitions are permitted when listing the sources of an $f$-polyhedron. Thus, $P$ --- as in Definition \ref{def_fpoly} --- has at most $n$ facets.  
\end{rmks}

Let $\Om$, $f$, $P$ and $J$ be as in Definition \ref{def_fpoly} above. We will use the following notation.
\begin{itemize}
	\item $\de(P):=\max\{\de_j:1\leq j\leq n\ \text{and}\ (\de_1,...,\de_n)\ \text{is the size-tuple of}\ P\}$.
	\item $\polcl{f}{n}:=$ the collection of all $f$-polyhedra in $\Om$ 
					with at most $n$ facets.
	\item $\polcl{J;f}{n}:=$ the collection of all $f$-polyhedra over $J$
					 with	at most $n$ facets.
	\item $\polcl{J\subset H;f}{n}	:=\{P\in\polcl{J;f}{n}:\bdy\Om\setminus \conj P\subset H\}$, where $H\subset\bdy\Om$ is a compact superset of $J$.
	\item $v(\Om;\mathscr{P}):=\inf
			\{\vol(\Om\setminus P):P\in\mathscr{P}\}$, for any sub-collection
			 $\mathscr{P}\subset\polcl{J\subset H;f}{n}$.
	\item  $v_n(f):=v(\Om;\polcl{J\subset H;f}{n})$, when the choice of $\Om$, $J$ and $H$ is unambiguous.
	\item  $v_n(J\subset H):=v(\Om;\polcl{J\subset H;f}{n})$, when the choice of $\Om$ and $f$ is unambiguous.
\end{itemize}

We now introduce some terminology and notation that will be used repeatedly in Section \ref{sec_localest}.
\begin{itemize}
\item Let $\rho:U\rightarrow\rl$ be $\cont^2$-smooth. The Levi polynomial associated to $\rho$ is the map $\levip_\rho:U\times U\rightarrow\C$ given by
	\bes
		\levip(z,w)=\sum_{j=1}^2\partl{\rho}{z_j}(w)(z_j-w_j)+\frac{1}{2}
					\sum_{j,k=1}^2\secpartl{\rho}{z_j}{z_k}(w)(z_j-w_j)(z_k-w_k).
	\ees
If the choice of $\rho$ is unambiguous, we will use $\levip$ instead.
\item Let $\rho:U\rightarrow\rl$ be $\cont^2$-smooth. The Cauchy-Leray map associated to $\rho$ is the map $\cler_\rho:U\times U\rightarrow\C$ given by
	\bes
		\cler_\rho(z,w)=\sum_{j=1}^2\partl{\rho}{z_j}(w)(z_j-w_j).
	\ees
\item $\mathcal{S}_\lam=\{(z_1,z_2)\in\CC:\rho^\lam(z_1,z_2)<0\}$, where $\rho^\lam(z_1,z_2)=\lam|z_1|^2-\ima z_2$. When $\lambda=1$, $\mathcal{S}_\lam=\mathcal{S}$.
\item For brevity, $\cler_\lam:=\cler_{\rho^\lam}$, and $f_\lam(z,w):=-2i\lam\cler_\lam(z,w)$ when $w\in\bdy\mathcal{S}^\lam$.
\item As defined in Theorem \ref{thm_MAIN}, for any domain $\Om\subset\CC$ and $\tau>0$, $\Om_\tau:=\{(z,w)\in\overline\Om\times\bdy\Om:||z-w||<\tau\}$. 

\end{itemize}
%------
%--Integral Kernel Stuff
%------
%%%%%%%%%%%%%%%%%%%%%%%%% 
%%%%%%%%%%%%%%%%%%%%%%%%% 
%%%%%%%%%%%%%%%%%%%%%%%%% Technical Lemmas %%%%%%%%%%%%%%%%%%%%%%%%%%%%
%%%%%%%%%%%%%%%%%%%%%%%%% 
%%%%%%%%%%%%%%%%%%%%%%%%% 
%\newpage
\section{Some Technical Lemmas}\label{sec_techlemm}
Here, we restrict our attention to Jordan measurable domains $\Om\subset\CC$. $J$ and $ H$ are compact subsets of $\bdy\Om$ such that $J\subset \inte_{\bdy\Om}H$. We will concern ourselves with $f$-polyhedra over $J$ that are constrained by $H$. We first prove a lemma that will allow us to work locally.

\begin{lemma}\label{lem_shrunbdd}Let $\Om$, $J$ and $H$ be as above. Suppose there are $\de_0>0$, $c>0$ and $f,g\in\cont(\conj\Om\times H)$ such that	 
	\begin{itemize}
		\item [$(a)$] $\{z\in\conj\Om:f(z,w)=0\}=\{z\in\conj\Om:g(z,w)=0\}=\{w\}$, for any fixed $w\in H$,  
		\item [$(b)$] $C(w,\de;f)\supseteq C(w,c\de;g)$, for all $w\in H$ and $\de<\de_0$, 
		\item[$(c)$] $C(w,\de;g)$ is Jordan measurable for each $w\in H$ and $\de<c\de_0$. 
	\end{itemize}
Then, for $P_n\in\polcl{J\subset H;f}{n}$ such that $\lim_{n\rightarrow \infty}\vol(\Om\setminus P_n)=0$, we have that $\lim_{n\rightarrow\infty}\de(P_n)= 0$. 
\end{lemma}
\begin{proof} It suffices to show that for each $\de<\de_0$, there is a $b>0$ such that $\vol(C(w,\de;f))>b$ for all $w\in H$. By condition $(b)$, it is enough to show this for the cuts of $g$. By $(a)$, $\vol\big(C(w,\de;g)\big)>0$ for each $w\in H$ and $\de<c\de_0$. Thus, it is enough to prove the continuity of $w\mapsto\vol\big(C(w,\de;g)\big)$, $\de<c\de_0$, on the compact set $H$.

Fix a $\de\in(0,c\de_0)$. Let $\chi_w:=\chi_{_{C(w,\de;g)}}$, where $\chi_{_A}$ denotes the indicator function of $A$. For a given $w\in H$, consider a sequence of points $\{w^n\}_{n\in\N}\subset H$ that converges to $w$ as $n\rightarrow \infty$. Then, 
	\be\label{eq_a.e.conv}
	\lim_{n\rightarrow\infty}\chi_{w^n}(z)=\chi_w(z)\hspace{2em} 
	\text{for a.e.}\ z\in\conj{\Om}.
	\ee
To see this, consider a $z\in\conj{\Om}$ such that $\chi_w(z)=0$. Suppose, there is a subsequence $\{w^{n_j}\}_{j\in\N}\subset\{w^n\}_{n\in\N}$ such that $\chi_{w^{n_j}}(z)=1$. Then, $|g(z,w^{n_j})|\leq \de$ but $\lim_{j\rightarrow\infty}|g(z,w^{n_j})|=|g(z,w)|\geq \de$. This is only possible if $g(z,w)=\de$. An analogous argument holds if $\chi_w(z)=1$. Thus, $z\in\bdy C(w,\de;g)$. Due to assumption $(c)$, this is a null set. Thus, \eqref{eq_a.e.conv} is true and we invoke Lebesgue's dominated convergence theorem to conclude that 	
	\bes
		\vol\big(C(w^n,\de;g)\big)=\int_{\conj\Om} \chi_{w^n}d\omega\xrightarrow{n\rightarrow\infty} 
		\int_{\conj\Om}\chi_w d\omega=\vol\big(C(w,\de;g)\big), 
	\ees
where $\de<c\de_0$ and $\omega=\vol$ is the Lebesgue measure on $\CC$.  
\end{proof}

%--------------------
%-------main lemma
%--------------------
Next, we prove a lemma that permits us to concentrate on a single representative of $\operatorname{LP}(\Om)$. 

\begin{lemma}\label{lem_main} 
Let $\Om$, $J$ and $H$ be as above. Suppose $f,g\in\cont(\conj\Om\times H)$ are such that
	\begin{itemize}
		\item [$(i)$]  $\{z\in\conj\Om:f(z,w)=0\}=\{z\in\conj\Om:g(z,w)=0\}=\{w\}$, for any fixed $w\in H$, and 
		\item [$(ii)$] there exist constants $\eps\in(0,1/3)$ and $\tau>0$, such 
							that
							\be\label{ineq_f-g}
								|f(z,w)-g(z,w)|
								\leq \eps(|g(z,w)+|f(z,w)|)
							\ee
				on $\{(z,w)\in\conj\Om\times H:||z-w||\leq\tau\}$. 
	\end{itemize}
Further, assume that the cuts of $g$ are Jordan measurable and satisfy a doubling property as follows 
				\begin{itemize}
					\item[\refstepcounter{equation}(\theequation)\label{eq_doub}] there is a $\de_g>0$ and a continuous $D:[0,16]\rightarrow\rl$ so that, for any 
					$n\in\pnat$, $(w^j,\de_j)\in H\times (0, \de_g),\ 1\leq j\leq n$, and $t\in[0,16]$, 
						\bes
							\vol\left(\bigcup_{j=1}^nC(w^j,(1+t)\de_j)\right)
							\leq D(t)\cdot \vol\left(\bigcup_{j=1}^nC(w^j,\de_j) 
							\right).
						\ees
			\end{itemize}
Then, for every $\beta>0$,
	\bea
		\limsup\limits_{n\rightarrow\infty} n^\beta v_n(f)
	&\leq& D\left(\frac{(1+\eps)^2}{(1-\eps)^2}-1\right)
		\limsup\limits_{n\rightarrow\infty} n^\beta v_n(g);\label{eq_f<=g}	\\
		\liminf\limits_{n\rightarrow\infty}n^\beta v_n(f)
	&\geq&D\left(\frac{(1+\eps)^4}{(1-\eps)^4}-1\right)^{-1}
		\liminf\limits_{n\rightarrow\infty}n^\beta v_n(g),
			\label{eq_f>=g}
	\eea
where $v_n(h)=v\left(\Om;\polcl{J\subset H;h}{n}\right)$, $D_1(\eps)=$ and $D_2(\eps)$.
\end{lemma}
%--------------Proof-----------------
\begin{proof} Observe that if $\heps:=\frac{1+\eps}{1-\eps}$, then inequality \eqref{ineq_f-g} may be transcribed as 
	\be\label{ineq_f-gnew} 
		|f(z,w)|\leq \hat\eps|g(z,w)|\ \text{and}\ 
		|g(z,w)|\leq \hat\eps |f(z,w)|
	\ee
 on  $\{(z,w)\in\conj\Om\times H:||z-w||\leq\tau\}$. Hence, for any $w\in H$ and $\de>0$,
	\bea 
		C(w,\de;f)\subseteq \bbb_2(w;\tau)&\Rightarrow& 
				C(w,\de;f)\subseteq
				C\left(w,\hat\eps\de;g\right); 
						\label{fing}\\
		C(w,\de;g)\subseteq \bbb_2(w;\tau)&\Rightarrow& 
				C(w,\de;g)\subseteq 
				C\left(w,\hat\eps\de;f\right).
						\label{ginf}
	\eea

%--------------Upper Bound-----------------
We first show that
	\bes 	
		\limsup\limits_{n\rightarrow\infty}n^\beta v_n(f)
		\leq
		D\left(\frac{(1+\eps)^2}{(1-\eps)^2}-1\right)
		\limsup\limits_{n\rightarrow\infty}n^\beta v_n(g).
	\ees
Let $\xi>1$. Assume that $L_{\sup}:=\limsup_{n\rightarrow\infty}n^\beta v_n(g)$ is finite. Then, there is an $n_\xi\in\pnat$ such that for each $n\geq n_\xi$, we can pick a $Q_n\in\polcl{J\subset H;g}{n}$ satisfying
	\be\label{ineq_gcutupp}
		\vol(\Om\setminus Q_n)\leq \xi L_{\sup} n^{-\beta}.
	\ee
As the cuts of $g$ are Jordan measurable, Lemma \ref{lem_shrunbdd} implies that $\de(Q_n)\rightarrow 0$ as $n\rightarrow\infty$. Consequently, $n_\xi$ can be chosen so that \eqref{ineq_gcutupp} continues to hold, and for all source-size pairs $(w,\de)$ of $Q_n$, $n\geq n_\xi$, we have that 
	\begin{itemize}
		\item [$(a)$] $\de<\de_g$ (see condition \eqref{eq_doub} on $g$);
		\item [$(b)$] $C(w,\de;g)\subset \bbb_2(w;\tau)$ 
							and $C(w,4\de;g)\cap \bdy\Om\subset H$; and 
		\item [$(c)$] $C(w,2\de;f)\subset \bbb_2(w;\tau)$. 
	\end{itemize}
The second part of $(b)$ is possible as each cut of $Q_n$ is compelled to intersect $J$ non-trivially, by definition. For a fixed source-size pair $(w,\de)$ of $Q_n$, we have, due to \eqref{ginf} and \eqref{fing},
	\bes\label{eq_incI}
		C(w,\de;g)\subseteq C(w,\heps\de;f)\subseteq C\left(w,\heps^2\de;g\right).  
	\ees
The second inclusion is valid as $\heps\de\leq 2\de$, thus permitting the use of \eqref{fing}, given  $(c)$. 

We can now approximate $Q_n$ by an $f$-polyhedron by setting
	\beas	
		\wt{Q_n}&:=&\left\{z\in\Om:
					|g(z,w)|>\heps^2\de, (w,\de)
					\ \text{is a source-size pair of}
				\ Q_n\right\};\\
		P_n&:=&\{z\in\Om:|f(z,w)|>\heps\de,(w,\de)
				\ \text{is a source-size pair of}\ Q_n\}.
	\eeas
Our assumptions imply that $\wt{Q_n}$ and $P_n$ are in $\polcl{J\subset H;g}{n}$ and $\polcl{J\subset H;f}{n}$, respectively. From the above inclusions, we have that $\wt{Q_n}\subseteq P_n\subseteq Q_n$, $n\geq n_\xi$. Hence, by property \eqref{eq_doub} of $g$ and \eqref{ineq_gcutupp}, we see that 
	\beas	
		n^\beta v_n(f)
		\leq n^\beta \vol(\Om\setminus P_n) 
		&\leq& n^\beta \vol\big(\Om\setminus \wt{Q_n}\big)\\ 
		&\leq & D\big(\heps^2-1\big) n^\beta \vol(\Om\setminus Q_n)
					\\
		&\leq& \xi D\big(\heps^2-1\big)L_{\sup},
	\eeas
for $n\geq n_\xi$. As $\xi>0$ was arbitrary and $\heps=\frac{1+\eps}{1-\eps}$, \eqref{eq_f<=g} follows.

%--------------Lower Bound-----------------
To complete this proof, we show that
	\bes 
		\liminf\limits_{n\rightarrow\infty}n^\beta v_n(f)
		\geq
		D\left(\frac{(1+\eps)^4}{(1-\eps)^4}-1\right)^{-1}
		\liminf\limits_{n\rightarrow\infty}n^\beta v_n(g).
	\ees
For this, fix a $\xi>1$, and assume that $L_{\inf}:=\liminf_{n\rightarrow\infty}n^\beta v_n(g)$ is finite. 
Thus, there is an $n_\xi\in\pnat$ such that 
	\be\label{ineq_gcutlow}
		v_n(g)\geq \frac{1}{\xi} L_{\inf} n^{-\beta};\ 
			\text{for}\ n\geq n_\xi.
	\ee
For each $n$, we pick an $R_n\in\polcl{J\subset H;f}{n}$ that satisfies
	\be\label{ineq_fglbi}
		v(\Om\setminus R_n)\leq \xi v_n(f).
	\ee
Now, we may also assume that $\liminf_{n\rightarrow\infty}n^\beta v_n(f)<\infty$ (else, there is nothing to prove), thus obtaining that $v_n(f)\rightarrow 0$ for infinitely many $n\in\pnat$. But, as $v_n(f)$ is decreasing in $n$, we get that  $v_n(f)\rightarrow 0$ for all $n\in\pnat$. Now, due to \eqref{ginf}, it is possible to choose $\de$ small enough so that
	\bes
		C\left(w,\frac{\de}{\heps};g\right)\subseteq C(w,\de;f),
	\ees
for each $w\in H$. As the cuts of $g$ are Jordan measurable (there is no such assumption on the cuts of $f$), we invoke Lemma \ref{lem_shrunbdd} to conclude that $\de(R_n)\rightarrow 0$ as $n\rightarrow\infty$. As before, we find a new $n_\xi$ such that \eqref{ineq_gcutlow} continues to hold, and for all $n\geq n_\xi$ and all source-size pairs $(w,\de)$ of $R_n$, we have  
	\begin{itemize}
			\item[$(a')$]  $\de<\de_g$ (see condition \eqref{eq_doub} on $g$); 
			\item[$(b')$]  $C(w,4\de;f)\subset \bbb_2(w;\tau)$ and
								$C(w,4\de;f)\cap\bdy\Om\subset H$; and
			\item[$(c')$] 	$C(w,2\de;g)\subset \bbb_2(w;\tau)$.
	\end{itemize}
Then, as before
	\be\label{eq_incII}
 		C\left(w,\frac{\de}{\heps};g\right)\subseteq C(w,\de;f)
		\subseteq	C\left(w,\heps{\de};g\right)
		\subseteq C\left(w,\heps^2\de;f\right)
		\subseteq C\left(w,\heps^3\de;g\right).
	\ee

We now approximate $R_n$ with an $n$-faceted $g$-polyhedron, using
		\beas	
		\wt{R_n}:& = &\left\{z\in\Om:|f(z,w)|>\heps^2\de, (w,\de)
					\ \text{is a source-size pair of}\ R_n\right\};\\
		S_n:& =& \left\{z\in\Om:|g(z,w)|>\heps{\de}, (w,\de)
					\ \text{is a source-size pair of}\ R_n\right\}.
	\eeas
Our assumptions are designed to ensure that $\wt{R_n}\in\polcl{J\subset H;f}{n}$ and $S_n\in\polcl{J\subset H;g}{n}$. From the above inclusions, we have that
	\bes
		\wt{R_n}\subseteq S_n\subseteq R_n,\ n\geq n_\xi.
	\ees
Moreover, the first and last inclusions in \eqref{eq_incII} and the assumption \eqref{eq_doub} on $g$ (note that $\heps^4<16)$ imply that
	\bea\label{ineq_fglbii}
		&&\vol\big(\Om\setminus \wt{R_n}\big)-\vol(\Om\setminus R_n)\notag \\
		&\leq& 
			\vol\left(\bigcup_{(w,\de)\in\Lambda_n}
			C\left(w,\heps^3\de;g\right)
				-\bigcup_{(w,\de)\in\Lambda_n}
			C\left(w,\frac{\de}{\heps};g\right)\right)\notag \\
			&\leq& D\big(\heps^4-1\big)\vol(\Om\setminus R_n),
	\eea
where $\Lambda_n$ is the set of source-size pairs of $R_n$. 
		
Therefore, using \eqref{ineq_fglbii} and \eqref{ineq_fglbi}, we see that  
	\beas	
		\frac{1}{\xi} L_{\inf}n^{-\beta} \leq v_n(g)
		&\leq&  \vol(\Om\setminus S_n)\leq \vol\big(\Om\setminus \wt{R_n}\big)\\
		&\leq&D\big(\heps^4-1\big)\vol(\Om\setminus R_n)\\
		&\leq& D\big(\heps^4-1\big)\xi v_n(f).
	\eeas	
Therefore,
	\bes
		n^\beta v_n(f)
		\geq \xi^{-2}
		 D\big(\heps^4-1\big)^{-1} L_{\inf},
		\ \ \ n\geq n_\xi.
	\ees
As $\xi>0$ was arbitrary and $\heps=\frac{1+\eps}{1-\eps}$, \eqref{eq_f>=g} follows.
\end{proof}

\begin{remark}\label{rmk_locglob} In practice, $f$ and $g$ may only be defined on $(\conj\Om\cap U)\times H$ for some open set $U\subset\CC$ containing a $\tau$-neighborhood of $H$, while satisfying the analogous version of condition $(i)$ there. As the remaining hypothesis (and indeed the result itself) depends only on the values of $f$ and $g$ on an arbitrarily thin tubular neighborhood of $H$ in $\conj\Om$, we may replace $f$ (and, similarly, $g$) by $f_{\sf e}$ to invoke Lemma \ref{lem_main}, where 
	\bes
		f_{\sf e}:=f(z,w)\varsigma(||z-w||^2)+||z-w||^2(1-\varsigma(||z-w||^2))
	\ees  
for some non-negative $\varsigma\in\cont^\infty(\rl)$ such that $\varsigma(x)=1$ when $x\leq \tau^2/2$ and $\varsigma(x)=0$ when $x\geq \tau^2$. We will do so without comment, when necessary.  
\end{remark}

%%%%%%%%%%%%%%%%%%%%%%%% 
%%%%%%%%%%%%%%%%%%%%%%%%% 
%%%%%%%%%%%%%%%%%%%%%%%%% Approx Model Domains %%%%%%%%%%%%%%%%%%%%%%%%%%%%
%%%%%%%%%%%%%%%%%%%%%%%%% 
%%%%%%%%%%%%%%%%%%%%%%%%% 
%\newpage
\section{Approximating Model Domains}\label{sec_approxmodel}
As a first step, we examine volume approximations of the Siegel domain by a particular class of analytic polyhedra. This problem enjoys a connection with Laguerre-type tilings of the Heisenberg surface equipped with the Kor{\'a}nyi metric (see the appendix for further details). 

Let $\mathcal{S}:=\{(z_1,x_2+iy_2)\in\CC:y_2>|z_1|^2\}$ and $f_\mathcal{S}(z,w)=z_2-\conj{w_2}-2 i z_1\conj{w_1}$. We view $\C\times\rl$ as the first Heisenberg group, $\bbh$, with group law
	\bes
		(z_1,x_2)\cdot_{\scriptscriptstyle\bbh}(w_1,u_2)=(z_1+w_1,x_2+u_2+2\ima (z_1\conj w_1))
	\ees
and the left-invariant Kor{\'a}nyi gauge metric (see \cite[Sec. 2.2]{CPDT}) 
	\bes
		d_{{\scriptscriptstyle\bbh}}((z_1,x_2),(w_1,u_2))
		:=||(w_1,u_2)^{-1}\cdot_{\scriptscriptstyle\bbh}(z_1,x_2)||_{\scriptscriptstyle\bbh},	
	\ees
where $||(z_1,x_2)||_{\scriptscriptstyle\bbh}^4:=|z_1|^4+x_2^2$. Observe that, for any cut $C(w,\de)=C(w,\de;f_\mathcal{S})$, $w\in\bdy \mathcal{S}$, $C(w,\de)'$ is the set
	\be\label{eq_proj_cuts}
		K(w',\sqrt\de)=\{(z_1,x_2)\in\C\times\rl:
		|z_1-w_1|^4+(x_2-u_2+2\ima(z_1\conj{w_1}))^2\leq\de^2\},
	\ee
which is the ball of radius $\sqrt\de$ centered at $w'$, in the Kor{\'a}nyi metric. 

\begin{ntn} We will use the following notation in this section:
	\begin{itemize}
		\item $I^r:=\{(x_1+i y_1,x_2)\in\C\times\rl:0\leq x_1\leq r,\ 0\leq y_1\leq r,\ 0\leq x_2\leq r^2\}$, $r>0$. 
		\item $\hat I^r:=I^{2r}-\left(\frac{r}{2}+i\frac{r}{2},\frac{3r^2}{2}\right)$, $r>0$. $I^r\subset\hat I^r$ and they are concentric.
		\item $v_n(J\subset H):=v(\mathcal{S};\polcl{J\subset H;f_\mathcal{S}}{n})$,
				 for $J\subset H\subset \bdy \mathcal{S}$. If $J\subset H\subset \C\times\rl$, $v_n(J\subset H)$ is meaningful in view of the obvious correspondence between $\C\times\rl$ and $\bdy \mathcal{S}$.
	\end{itemize}
\end{ntn}	

\begin{lemma}\label{lem_heis}
Let $I=I^{1}$ and $\hat I={\hat I}^{1}$. There exists a positive constant $\lkor>0$ such that 
	\bes
		v_n(I\subset \hat I)\sim \frac{\lkor}{\sqrt{n}}
	\ees
as $n\rightarrow\infty$. 
\end{lemma}
\begin{proof} Simple calculations show that
	\bea
		&&\vol(C(w,\de))=\frac{2\pi}{3}\de^3\label{eq_vol}\\
		&&\rvol(K(w',\sqrt\de))=\frac{\pi^2}{2}\de^2 \label{eq_volkor}
	\eea
for all $w\in\bdy \mathcal{S}$ and $\de>0$.

We utilize a special tiling in $\C\times\rl$. Let $k\in\pnat$ and consider the following points in $\C\times\rl$: 
	\bes
		v_{pqr}:=\left(\frac{p}{k}+i\frac{q}{k},\frac{r}{k^2}\right),\ (p,q,r)\in\Sigma_k,
	\ees
where $\Sigma_k:=\left\{(p,q,r)\in\mathbb{Z}^3:-2q\leq r\leq k^2-1+2p,\ 0\leq p,q\leq k-1\right\}$. Observe that $\operatorname{card}(\Sigma_k)=k^4+2k^3-2k^2$. Now, we set $E_{pqr}:= v_{pqr}\cdot_{\scriptscriptstyle\bbh}I^{\frac{1}{k}}$, and note that 
$I\subset\cup_{\Sigma_k}E_{pqr}\subset \hat I$, for all $k\in\pnat$.

\begin{figure}[H]
\centering
\resizebox{2.2in}{!}{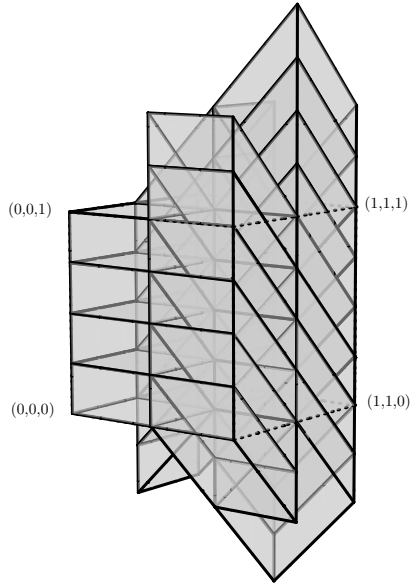}
\caption{The $24$ tiles $E_{pqr}$ when $k=2$.}
\end{figure} 

%-----------
%----PART I
%-----------
$\bm{1.}$ We first show that there is a constant $\alpha_1>0$ such that
	\be\label{heis_uppbd}
		v_n(I\subset \hat I)\leq\frac{\alpha_1}{\sqrt{n}}
	\ee
for all $n\in\pnat$.
 
For this, let 
	\bes
		u_{pqr}:=\text{center of}\ E_{pqr}=v_{pqr}\cdot_{\scriptscriptstyle\bbh}
		\left(\frac{1}{2k}+i\frac{1}{2k},\frac{1}{2k^2}\right) ,\ (p,q,r)\in\Sigma_k,\ k\in\pnat.
	\ees
Then, the Kor{\'a}nyi ball $K\left(u_{pqr},\frac{\sqrt[4]{5}}{\sqrt[4]{2}k}\right)$ (see \eqref{eq_proj_cuts}) contains $E_{pqr}$ and is contained in $\hat I$. Hence, if $w_{pqr}\in \bdy \mathcal{S}$ is such that ${w_{pqr}}'=u_{pqr}$, the cuts
	\bes
		C\left(w_{pqr},\frac{\sqrt{5}}{\sqrt{2}k^2};f_\mathcal{S}\right),\ (p,q,r)\in\Sigma_k,
	\ees
define $P_k$, an $f_\mathcal{S}$-polyhedron with $k^4+2k^3-2k^2$ facets. In fact, $P_k\in\polcl{I\subset \hat I;f_\mathcal{S}}{k^4+2k^3-2k^2}$, for all $k\in\pnat$, where we identify $I$ and $\hat I$ with their images in $\bdy S$ under the map $(z_1,x_2)\mapsto (z_1,x_2+i|z_1|^2)$. Therefore, using \eqref{eq_vol}
	\beas
		v_{k^4+2k^3-2k^2}(I\subset\hat I)&\leq& \vol(\mathcal{S}\setminus P_k)\\
		&\leq& 
		\vol\left(\bigcup_{\Sigma_k}		C\left(w_{pqr},\frac{\sqrt{5}}{\sqrt{2}k^2}\right)\right)\\
		&\leq& \frac{2\pi}{3}\left(\frac{\sqrt{5}}{\sqrt{2}k^2}\right)^3(k^4+2k^3-2k^2)
		=\frac{5\sqrt{5}\pi}{3\sqrt{2}}
		\frac{(k^4+2k^3-2k^2)}{k^6},
	\eeas
$ k\in\pnat$. Now, for a given $n\in\pnat$, choose $k$ such that $k^4+2k^3-2k^2\leq n\leq (k+1)^4+2(k+1)^3-2(k+1)^2$. Then, one can easily find a $\alpha_1>0$ such that
	\beas
		v_n(I\subset \hat I)\sqrt{n}
		&\leq&v_{k^4+2k^3-2k^2}(I\subset\hat I)\sqrt{(k+1)^4+2(k+1)^3-2(k+1)^2}\\
		&\leq&\frac{5\sqrt{5}\pi}{3\sqrt{2}}
		\frac{(k^4+2k^3-2k^2)\sqrt{(k+1)^4+2(k+1)^3-2(k+1)^2}}{k^6}\\
		&\leq& \alpha_1.
	\eeas

%-----------
%----PART II
%-----------
$\bm{2.}$ Next, we show that there is an $\alpha_2>0$ such that
	\be\label{heis_lowbd}
		v_n(I\subset \hat I)\geq\frac{\alpha_2}{\sqrt{n}}
	\ee
for $n\in\pnat$.

If finitely many Kor{\'a}nyi balls of radii $\sqrt{\rho_1},...,\sqrt{\rho_k}$ cover $I$, then \eqref{eq_volkor} yields
	\be\label{eq_covK}
		(\sqrt\rho_1)^4+\cdots+(\sqrt\rho_k)^4\geq \frac{2}{\pi^2}\rvol(I)=\frac{2}{\pi^2}.
	\ee

We will also need the following mean inequality (a consequence of Jensen's inequality)
	\be\label{eq_mean}
		\left(\frac{\rho_1^{d+1}+\cdots+\rho_k^{d+1}}{k}\right)^\frac{1}{d+1}
		\geq\left(\frac{\rho_1^{d-1}+\cdots+\rho_k^{d-1}}{k}\right)^\frac{1}{d-1},
	\ee
for positive $\rho_j$, $1\leq j\leq k$, and $d>1$. 

Now, fix a $\xi>1$. Let $P_n\in\polcl{I\subset \hat I;f_\mathcal{S}}{n}$ be such that
	\be\label{eq_lowerbound}
		\vol(\mathcal{S}\setminus P_n)\leq \xi v_n(I\subset \hat I).
	\ee
Let $C_j$ and $K_j$, $j=1,...,n$, be the cuts and their projections, respectively, of $P_n$. Now, $\K_n:=\{K_j,j=1,...,n\}$ is a finite covering of $I$, so by the Wiener covering lemma (see \cite[Lemma 4.1.1]{KrPa} for a proof that generalizes to metric spaces), we can find, on renumbering the indices, disjoint Kor{\'a}nyi balls $K_1,...,K_k\in\K_n$ of radii $\sqrt{\rho_1},...,\sqrt{\rho_k}$, such that $\cup_{K\in\K_n}K\subset\cup_{1\leq j\leq k}3 K_j$, where, for $j=1,...,k$, $3 K_j$ has the same centre as $K_j$ but thrice its radius. Let $C_j$ denote the cut that projects to $K_j$, $j=1,...,k$. It follows from \eqref{eq_lowerbound}, \eqref{eq_vol} and the inequalities \eqref{eq_mean} (for $d=5$) and \eqref{eq_covK} that
	\beas
		v_n(I\subset \hat I)\sqrt{n}&\geq& \frac{1}{\xi} \vol\left(\bigcup_{j=1}^k C_j\right)\sqrt{k}\\
		&=& \frac{1}{\xi} \left(\sum_{i=1}^k\vol(C_j)\right)\sqrt{k}
		= \frac{2\pi}{3\xi}\left(\rho_1^3+\cdots+\rho_k^3\right)\sqrt{k}\\
		&=& \frac{2\pi}{3^7\xi}\left((9\rho_1)^3+\cdots+(9\rho_k)^3\right)\sqrt{k}
		= \frac{2\pi}{3^7\xi}\left((3\sqrt{\rho_1})^6+\cdots+(3\sqrt{\rho_k})^6\right)
		k^{\frac{2}{4}}\\
		&\geq& \frac{2\pi}{3^7\xi}\left((3\sqrt{\rho_1})^4+\cdots+(3\sqrt{\rho_k})^4\right)^{\frac{6}{4}}\\
		&\geq& \frac{4\sqrt{2}}{\pi^2 3^7\xi}\rvol(I)^{\frac{3}{2}}
				=\frac{4\sqrt{2}}{\pi^2 3^7\xi}>0,\ 
		\text{for}\ n=n_0,n_0+1,... 		
	\eeas
As $\xi>1$ was arbitrary, we have proved \eqref{heis_lowbd}.

%-----------
%----PART III
%-----------
$\bm{3.}$ Define
	\bes
		 \lkor=
		\liminf_{n\rightarrow\infty} v_n(I\subset \hat I)\sqrt n.
	\ees
By \eqref{heis_lowbd} and \eqref{heis_uppbd}, $0<\lkor<\infty$. We now show that
	\be\label{heis_fin}
		 \lkor=\lim_{n\rightarrow\infty}
		v_n(I\subset \hat I)\sqrt{n}.
	\ee
For this, it suffices to show that for every $\xi>1$, if $n_0\in\pnat$ is chosen so that 
	\be\label{assump}
		v_{n_0}(I\subset\hat I)\sqrt{n_0}\leq \xi \lkor
	\ee
then, 
	\be\label{conc}
		v_n(I\subset \hat I)\sqrt{n}\leq \xi^4  \lkor
	\ee
for $n$ sufficiently large. 

Now, let $P_{n_0}\in\polcl{I\subset \hat I;f_\mathcal{S}}{n_0}$ be such that
	\bes
		\vol(\mathcal{S}\setminus P_{n_0})\leq \xi v_{n_0}(I\subset\hat I).
	\ees
For any $w\in\bdy \mathcal{S}$ and $k\in\pnat$, let $A_{w,k}:\CC\rightarrow\CC$ be the biholomorphism 
	\bes
(z_1,z_2)\mapsto 
	\left(w_1+\frac{1}{k}z_1,w_2+\frac{1}{k^2}z_2-\frac{2i}{k}z_1\conj{w_1}\right).
	\ees
 Then, $A_{w,k}$ has the following properties:
	\begin{itemize}
		\item $A_{w,k}^{\operatorname{res}}(z')=w'\cdot_{\scriptscriptstyle\bbh} (\frac{1}{k}z_1,
\frac{1}{k^2}x_2)$;
		\item $A_{w,k}(\mathcal{S})=\mathcal{S}$;
		\item $A_{w,k}(P_{n_0})\in\polcl{w' \cdot_{ \scriptscriptstyle\bbh} I^{\frac{1}{k}}\subset w'\cdot_{\scriptscriptstyle\bbh} \hat I^{\frac{1}{k}};f_\mathcal{S}}{n_0}$; and
		\item $\vol(\mathcal{S}\setminus A_{w,k}(P_{n_0}))\leq\xi\frac{v_{n_0}(I\subset\hat I)}{k^6}$.
	\end{itemize}
As a consequence,
	\bes
		P:=\bigcup_{(p,q,r)\in\Sigma_k} A_{{v_{pqr}},k}(P_{n_0})
	\ees
satisfies the following conditions:		
	\begin{itemize}		
		\item $P\in\polcl{I\subset \hat I;f_\mathcal{S}}{n_0(k^4+2k^3-2k^2)}$
		\item $\vol(\mathcal{S}\setminus P)\leq\xi v_{n_0}(I\subset\hat I)\frac{k^4+2k^3-2k^2}{k^6}$.
	\end{itemize}
Hence, by assumption \eqref{assump},
	\bea\label{eq_specseq}
		v_{n_0(k^4+2k^3-2k^2)}(I\subset \hat I)\sqrt{n_0(k^4+2k^3-2k^2)}
		&\leq& \xi v_{n_0}(I\subset\hat I)\sqrt{n_0}\frac{(k^4+2k^3-2k^2)^{\frac{3}{2}}}{k^6}\notag \\
		&\leq&\xi^2 v_{n_0}(I\subset\hat I)\sqrt{n_0}\leq \xi^3 \lkor,
	\eea
for sufficiently large $k$. Choose $k_0$ so that \eqref{eq_specseq} holds and $\frac{(k+1)^4+2(k+1)^3-2(k+1)^2}{k^4+2k^3-2k^2}\leq \xi^2$ for $k>k_0$. For $n\geq n_0(k_0^4+2k_0^3-2k_0^2)$, let $k$ be such that $n_0(k^4+2k^3-2k^2)\leq n\leq n_0((k+1)^4+2(k+1)^3-2(k+1)^2)$. Consequently,
	\beas
		v_n(I\subset \hat I)\sqrt{n}
		&\leq& v_{n_0(k^4+2k^3-2k^2)}(I\subset \hat I)\sqrt{n_0((k+1)^4+2(k+1)^3-2(k+1)^2)}\\
		&\leq& \xi^3\lkor\sqrt{\frac{(k+1)^4+2(k+1)^3-2(k+1)^2}{k^4+2k^3-2k^2}}\leq \xi^4  \lkor,
	\eeas
by \eqref{eq_specseq}. We have proved \eqref{conc} and, therefore, our claim \eqref{heis_fin}.
\end{proof}

%%%%%%%%%%%%%%%%%%%%%%%%% General Jordan Meas Sets %%%%%%%%%%%%%%%%%%%%%%%%%%%%
Our choice of the unit square in the above lemma facilitates the computation for polyhedra lying above more general Jordan measurable sets in the boundary of $\mathcal{S}$.

\begin{lemma}\label{heis_mod} Let $J,H\subset\bdy \mathcal{S}$ be compact and Jordan measurable with $J\subset \inte_{\bdy \mathcal{S}}H$. Then
	\bes
		v_n(J\subset H) \sim \rvol(J')^{\frac{3}{2}} \lkor\frac{1}{\sqrt{n}}
	\ees
as $n\rightarrow\infty$.
\end{lemma}
\begin{proof} 
%-----------
%----PART I
%-----------
$\bm{1.}$ We first show that 
	\be\label{heis_mod_lowbd}
		\limsup_{n\rightarrow \infty}v_n(J\subset H)\sqrt{n}\leq  \lkor\rvol(J')^{\frac{3}{2}}.
	\ee

Let $\xi>1$ be fixed. As $J$ is Jordan measurable, we can find $m$ points $v^1,...,v^m\in\C\times\rl$ and some $r>0$, such that 
	\be\label{J_cov}
		J'\subset\bigcup_1^m \left(v^j\doth I^r\right)\subset\bigcup_1^m \left(v^j \doth \hat I^r\right)\subset H'
	\ee
and
	\be\label{J_vol}
		m\rvol(I^r)\leq \xi\rvol(J').
	\ee
Now, observe that 
	\be\label{dil} 
		\sqrt{k}\ \frac{v_k(v^j\doth I^r\subset  v^j \doth \hat I^r)}{\rvol(I^r)^\frac{3}{2}}
		=\sqrt{k}\ \frac{v_k(I^r\subset \hat I^r)}{{\rvol(I^r)^\frac{3}{2}}}
		=\sqrt{k}\ v_k(I\subset \hat I).
	\ee 
Thus, due to \eqref{J_cov}, Lemma \ref{lem_heis}, \eqref{dil} and \eqref{J_vol}, we have
	\bea
		v_{km}(J\subset H)\sqrt{km} &\leq & \sum_{j=1}^{m}
		v_k(v^j\doth I^r \subset v^j\doth \hat I^r )\sqrt{k}\sqrt{m}\notag\\
		&\leq& \xi \lkor \rvol(I^r)^\frac{3}{2}m^{\frac{3}{2}}
		\label{comp}\\
		&\leq & \xi^{\frac{5}{2}}\lkor\rvol(J')^{\frac{3}{2}}\notag
	\eea
for $k$ sufficiently large. Choose $k_0\in\pnat$ such that for $k\geq k_0$, \eqref{comp} holds and $\sqrt{(k+1)/k}\leq \xi$. For sufficiently large $n$, we can find a $k\geq k_0$ such that $mk\leq n\leq m(k+1)$. Hence,
	\beas		
		v_n(J\subset H)\sqrt{n} &\leq& v_{km}(J\subset H)\sqrt{(k+1)m}\\
							  &\leq & \xi^{\frac{5}{2}} \lkor\rvol(J')^{\frac{3}{2}}
								\sqrt\frac{k+1}{k}\\
								&\leq & \xi^{\frac{7}{2}}\lkor\rvol(J')^{\frac{3}{2}}.
	\eeas
As $\xi>1$ was arbitrarily fixed, we have proved \eqref{heis_mod_lowbd}.

%-----------
%----PART II
%-----------
$\bm{2.}$ It remains to show that
	\be\label{heis_mod_uppbd}
		\liminf_{n\rightarrow \infty} v_n(J\subset H) \sqrt{n}\geq  \lkor \rvol(J')^{\frac{3}{2}}.
	\ee

Once again, fix a $\xi>1$. The Jordan measurability of $J$ ensures that there are pairwise disjoint sets $I_1,...,I_m$, where $I_j=v^j \doth I^{r_j}$ for some $r_j>0$ and $v^j\in\C\times\rl$, $1\leq j\leq m$, such that 
	\be\label{J_cont}	
		\bigcup_{1}^m I_j\subset J'\ \text{and}\ \bigcup_{1}^m {\hat I_j}\subset J',
	\ee
where $\hat I_j=v^j \doth \hat I^{r_j}$, and
	\be\label{J_vol2}
		\rvol(J')\leq \xi \sum_{j=1}^{m}\rvol(I_j).
	\ee 
Choose a $P_n\in\polcl{J\subset H;f_\mathcal{S}}{n}$ such that $v(\mathcal{S}\setminus P_n)\leq \xi v_n(J\subset H)$ and let $n_j$ denote the number of cuts of $P_n$ whose projections intersect $I_j$ and are contained in $\hat I_j$. By part $\bm{1.}$, $v_n(J\subset H)\rightarrow 0$ as $n\rightarrow\infty$. Thus, recalling \eqref{eq_vol}, $\de(P_n)\rightarrow 0$ as $n\rightarrow\infty$. So, we may choose $n$ so large that the projections of these $n_j$ cuts, in fact, cover $I_j$ and no two cuts of $P$ whose projections intersect two different $I_j$'s intersect. Therefore,
	\be\label{count_fac}
		n_1+\cdots +n_m\leq n.
	\ee
By Lemma \ref{lem_heis} and \eqref{dil}, there is an $n_0\in\pnat$ such that 
	\be\label{sq_lowbd}
		v_k(I_j\subset \hat I_j)\geq \frac{1}{\xi} \lkor \rvol(I_j)^{\frac{3}{2}}\frac{1}{\sqrt{k}}
	\ee
for $k\geq n_0$ and $j=1,...,m$. We may further increase $n$ to ensure that
	\bes	
		n_j\geq n_0 \ \text{for}\ j=1,...,m.
	\ees
Consequently, by \eqref{J_cont} and \eqref{sq_lowbd}, we have,
	\bes
		v_n(J\subset H)\geq \frac{1}{\xi}\sum_{j=1}^m v_{n_j}(I_j\subset \hat I_j)
		\geq \frac{ \lkor}{\xi^2}\sum_{j=1}^m \frac{\rvol(I_j)^{\frac{3}{2}}}{\sqrt{n_j}}.
	\ees
Now, H{\"o}lder's inequality yields,
	\bes	
		\sum_{j=1}^m\rvol(I_j)
		=\sum_{j=1}^m\left(\frac{\rvol(I_j)}{n_j^{1/3}}\right)n_j^{1/3} 
		\leq \left(\sum_{j=1}^m\frac{\rvol(I_j)^{3/2}}{n_j^{1/2}}\right)^{\frac{2}{3}}
			\left(\sum_{j=1}^m n_j\right)^{\frac{1}{3}}.
	\ees
Using this, \eqref{J_vol2} and \eqref{count_fac}, we obtain	
	\bes
		v_n(J\subset H) \geq
		\frac{ \lkor}{\xi^2}\left(\sum_{j=1}^m\rvol(I_j)\right)^{\frac{3}{2}}
		\left(\frac{1}{\sum_1^m n_j}  \right)^{\frac{1}{2}}
		\geq  \frac{ \lkor}{\xi^{7/2}}\rvol(J')^{\frac{3}{2}}\frac{1}{\sqrt{n}}
	\ees
for $n$ sufficiently large. As the choice of $\xi>1$ was arbitrary, \eqref{heis_mod_uppbd} now stands proved. 
\end{proof}

As a final remark, we extend the above lemma to a class of slightly more general model domains in order to illustrate the effect of the Levi determinant on our asymptotic formula.  

\begin{cor}\label{cor_levi} Let $\mathcal{S}_\lam:=\{(z_1,x_2+iy_2)\in\CC:y_2>\lam|z_1|^2\}$ and $f_{\mathcal{S}_\lam}(z,w)=\lam(z_2-\conj{w_2})-2 i \lam^2(z_1\conj{w_1})$, for $\lam>0$. Let $J,H\subset\bdy \mathcal{S}_\lam$ be compact and Jordan measurable with $J\subset \inte_{\bdy \mathcal{S}_\lambda}H$. Then
	\bes
		v_n(\mathcal{S}_\lam;J\subset H):=v(\mathcal{S}_\lam;\polcl{J\subset H;f_{\mathcal{S}_\lam}}{n}) 
		\sim \lam^{\frac{1}{2}}\rvol(J')^{\frac{3}{2}}{\lkor}\frac{1}{\sqrt{n}}
	\ees
as $n\rightarrow\infty$.
\end{cor}
 \begin{proof} Let $\Xi:\CC\rightarrow\CC$ be the biholomorphism $\Xi:(z_1,z_2)\mapsto(\lam z_1,\lam z_2)$. Then, $\mathcal{S}=\Xi(\mathcal{S}_\lam)$ and $f_{\mathcal{S}_\lam}(z,w)=f_\mathcal{S}(\Xi(z),\Xi(w))$. Therefore, there is a bijective correspondence between $\polcl{J\subset H;f_{\mathcal{S}_\lam}}{n}$ and $\polcl{\Xi J\subset \Xi H;f_\mathcal{S}}{n}$ given by $P\mapsto \Xi P$. Now, as $\det(\JacR\Xi)\equiv\lam^{4}$ and $\det(\JacR\Xi^{\operatorname{res}})\equiv\lam^{3}$, we have
	 \bes
		\frac{v_n(\mathcal{S}_\lam;J\subset H)}{\rvol(J')^{\frac{3}{2}}}
		=\frac{\lam^{-4}v_n(\mathcal{S};\Xi J\subset \Xi H)}
			{\lam^{-\frac{9}{2}}\rvol(\Xi^{\operatorname{res}}J')^{\frac{3}{2}}}
		\sim \lam^{\frac{1}{2}}\lkor\frac{1}{\sqrt{n}}.
	\ees
\end{proof}

%%%%%%%%%%%%%%%%%%%%%%%%% 
%%%%%%%%%%%%%%%%%%%%%%%%% 
%%%%%%%%%%%%%%%%%%%%%%%%% Local Picture %%%%%%%%%%%%%%%%%%%%%%%%%%%%
%%%%%%%%%%%%%%%%%%%%%%%%% 
%%%%%%%%%%%%%%%%%%%%%%%%% 
\section{Local Estimates Via Model Domains}\label{sec_localest}
Lemma \ref{lem_main} suggests a way to locally compare the volume-minimizing approximations drawn from two different classes of $f$-polyhedra which exhibit some comparability. In this section, our main goal is Lemma \ref{lem_locests} where we set up a local correspondence between $\Om$ and a  model domain $\mathcal{S}_\lam$, pull back the special cuts given by $f_{\mathcal{S}_\lam}$ (defined in Section \ref{sec_approxmodel}) via this correspondence, and establish a \eqref{ineq_f-g}-type relationship between the pulled-back cuts and those coming from the Levi polynomial of a defining function of $\Om$. First, we note a useful estimate on the Levi polynomial.
 
%-------------
%-------------
%-Prelim Est--
%-------------
%-------------
\begin{lemma}\label{lem_Leviest} Let $\Om$ be a $\cont^2$-smooth strongly pseudoconvex domain. Suppose $\rho\in\cont^2(\CC)$ is a strictly plurisubharmonic defining function of $\Om$. Then, there exist constants $c>0$ and $\tau>0$ such that 
	\be\label{eq_Leviest}
		||z-w||^2\leq c |\mathfrak{p}(z,w)|,
	\ee 
on $\Om_\tau=\{(z,w)\in\pkdom{\Om}:||z-w||\leq \tau\}$, where $\mathfrak{p}$ is the Levi polynomial of $\rho$. 
\end{lemma}
\begin{proof} The second-order Taylor expansion of $\rho$ about $w\in\bdy\Om$ gives:
	\bes
		-2\rea \mathfrak{p}(z,w)=-\rho(z)+\sum_{j,k=1}^2
		\secpartl{\rho(w)}{z_j}{\conj{z_k}}(z_j-w_j)(\conj{z_k}-\conj{w_k})
		+o(||z-w||^2),
	\ees
The strict plurisubharmonicity of $\rho$ implies the existence of a $c'>0$ so that 	
	\bes
		\sum_{j,k=1}^2
		\secpartl{\rho(w)}{z_j}{\conj{z_k}}(z_j-w_j)(\conj{z_k}-\conj{w_k})
		\geq c'	||z-w||^2,\ \  (z,w)\in\conj\Om\times\conj\Om.
	\ees
The result follows quite easily from this.  
\end{proof}

%-------------
%-------------
%---Darboux---
%-------------
%-------------
\noindent {\bf Special Darboux Coordinates.} As we are now going to construct a non-holomorphic transformation, we need to alternate between the real and complex notation. Here are some clarifications.
\begin{itemize}
\item We will use $z$ (and similarly $w$) to denote both $(z_1,z_2)=(x_1+iy_1,x_2+iy_2)\in\CC$ and $(x_1,y_1,x_2,y_2)\in\rl^4$. The usage will be clear from the context. In the same vein, by $z'$ we mean either $(z_1,x_2)=(x_1+iy_1,x_2)\in\C\times\rl$ or $(x_1,y_1,x_2)\in\rl^3$.
\item For any map $\Psi:\CC\rightarrow\CC$ and $z,w\in\CC$, $\JacR\Psi(w)(z-w)$ will either denote a vector in $\CC$ or a vector in $\rl^4$ depending on the context. Recall that $\JacR\Psi(w)(z-w)=\JacC\Psi(w)(z-w)+\operatorname{Jac}_{\bar\C\!}\Psi(w)(\overline z-\bar w)$, where $\operatorname{Jac}_{\bar\C\!}\Psi(w)$ is the matrix of complex conjugate derivatives of $\Psi$ at $w$. 
\item Recall that $\lbr \theta,z\rbr$ denotes the pairing between a complex covector and a complex vector. When $\theta$ is a real covector, we use $\ldbr\theta,z\rdbr$ to stress that $z$ is being viewed as a tuple in $\rl^4$.
\end{itemize}

 Fix a $\lam>0$. For reasons that will become clear in the next part of this section, we consider a special $\cont^4$-smooth strongly pseudoconvex domain $\Om$ such that $0\in\bdy\Om$ and for a neighborhood $U$ of the origin, there is a convex function function $\rho:U\rightarrow\rl$ such that $\Om\cap U=\{z\in U:\rho(z)<0\}$ and
	\be\label{eq_cvxdef2}
		\rho(z)=-\ima z_2+\lam |z_1|^2+2\rea(\mu z_1\conj{z_2})+\nu|z_2|^2+o(|z|^2).
	\ee
\sloppy
We may shrink $U$ to find a convex function $F:=F_\rho:U'\rightarrow \rl$ that satisfies  $\rho(z_1,x_2,F(z_1,x_2))=0$. $\rho$ and $F_\rho$ are both 	$\cont^4$-smooth and $-i(\partial\rho-\conj\partial\rho)$ is a $\cont^3$-smooth contact form on $\bdy\Om\cap U$. The domain $\mathcal{S}_\lam$ from Section \ref{sec_approxmodel} is such a domain with $\rho^\lam(z)=-\ima z_2+\lam|z_1|^2$ and $F_{\rho^\lam}(z_1,x_2)=\lam|z_1|^2$. 

Darboux's theorem in contact geometry (see \cite[Appendix 4]{Ar89}) says that any two equi-dimensional contact structures are locally contactomorphic. We seek local diffeomorphisms between $\Om$ and $\mathcal{S}_\lam$ that extend to local contactomorphisms between $(\bdy\Om,-i(\partial\rho-\conj\partial\rho))$ and $(\bdy \mathcal{S}_\lam,-i(\partial\rho^\lam-\conj\partial\rho^\lam))$, and satisfy estimates essential to our goal. We carry out this construction over the next three lemmas, working intially on $\rl^3$ instead of $\bdy\Om$. For this, if $\mathsf{gr}_\rho:U'\rightarrow U$ maps $(x_1,y_1,x_2)$ to $(x_1,y_1,x_2,F_\rho(x_1,y_1,x_2))$, we set
	\beas
		\theta_\rho
		&:=&(\mathsf{gr}_\rho)^*\left(\frac{\partial\rho-\conj\partial\rho}{i}\right)\\
		&=&\frac{-1}{\rho_{y_2}}\Big((\rho_{y_2}\rho_{y_1}+\rho_{x_1}\rho_{x_2})dx_1
		-(\rho_{y_2}\rho_{x_1}-\rho_{y_1}\rho_{x_2})dy_1
		+(\rho_{y_2}^2+\rho_{x_2}^2)dx_2\Big),
	\eeas
where, by the partial derivatives of $\rho$ we mean their pull-backs to $U'$ via $\mathsf{gr}_\rho$. 

%----------
%---- Contact Transf Lemma----
%----------
\begin{lemma}\label{lem_darb}
Let $\Om$ be defined by \eqref{eq_cvxdef2}. There is an open subset $(0\in)V\subset U'\subset\rl^3$ and a $\cont^2$-smooth diffeomorphism $\Pi=(\pi_1,\pi_2,\pi_3):V\rightarrow\rl^3$ with $\Pi(0)=0$ satisfying
	\begin{itemize}
		\item $\Pi^*\theta_{\rho^\lam}(z')=\alpha(z')\theta_\rho(z')$ for all 
				$z'\in V$, and some $\alpha\in\cont(V)$ with $\alpha(0)=1$; and
		\item $|\det\JacR\Pi(0)|=1$.
	\end{itemize}
\end{lemma}
\begin{proof} We proceed with the understanding that when refering to functions defined a priori on $U$ (such as $\rho$ or its derivatives) we implicitly mean their pull-backs to $U'$ via $\mathsf{gr}_\rho$. 
 
Now, consider the following $\cont^3$-smooth vector field in $\ker\theta_\rho$ on $U'$:
	\bes
		v=\partl{\rho}{x_2}\partl{}{x_1}
					-\partl{\rho}{y_2}\partl{}{y_1}
					-\partl{\rho}{x_1}\partl{}{x_2}.
	\ees
We let $\gamma^t(z'):=\gamma(z';t)=(\gamma_1(z';t),\gamma_2(z';t),\gamma_3(z';t))$ be the flow of $v$ such that $\gamma(z';0)=z'$. Note that $\gamma(z';t)$ is $\cont^3$-smooth in $z'$ and $\cont^4$-smooth in $t$. Differentiating the initial value problem for the flow, we have 
	\be\label{eq_hdergam}
		\JacR\gamma^0\equiv \id\ \text{and}\ \Hess_\rl\gamma^0\equiv 0.
	\ee 

Observe that the map 
	\bes
	\Gamma=(\Gamma_1,\Gamma_2,\Gamma_3):z'=(x_1,y_1,x_2)\mapsto\gamma(x_1,0,x_2;y_1), 
	\ees  
is defined on some neighborhood, $U'_1\subset U'$, of the origin. Moreover, dropping the arguments, switching to our shorthand notation, and denoting $f\circ\Gamma$ by $\wt{f}$, we have
	\beas
		\JacR\Gamma=
			\left( \begin{array}{ccc}
				{\Gamma_1}_{x_1} 
								&\wt{{\rho}_{x_2}}
											&{\Gamma_1}_{x_2}
			\vspace{0.5em} \\ 
				{\Gamma_2}_{x_1}
								& -\wt{{\rho}_{y_2}}
											&{\Gamma_2}_{x_2} 
			\vspace{0.5em} \\	
				{\Gamma_3}_{x_1} 
								& -\wt{{\rho}_{x_1}} 
											&{\Gamma_3}_{x_2} 
			\end{array} \right),
	\eeas
and
	\beas
		(\JacR\Gamma)^{-1}=
			\left( \begin{array}{ccc}
				\frac{\wt{\rho_{x_1}}{\Gamma_2}_{x_2}
					-\wt{\rho_{y_2}}{\Gamma_3}_{x_2}}{\det\JacR\Gamma} &
				\frac{-\wt{\rho_{x_1}}{\Gamma_1}_{x_2}
					-\wt{\rho_{x_2}}{\Gamma_3}_{x_2}}{\det\JacR\Gamma} &
				\frac{\wt{\rho_{y_2}}{\Gamma_1}_{x_2}
					+\wt{\rho_{x_2}}{\Gamma_2}_{x_2}}{\det\JacR\Gamma} 
		\vspace{0.5em} \\ 
				\frac{{\Gamma_2}_{x_2}{\Gamma_3}_{x_1}
					-{\Gamma_2}_{x_1}{\Gamma_3}_{x_2}}{\det\JacR\Gamma} &
				\frac{-{\Gamma_1}_{x_2}{\Gamma_3}_{x_1}
					+{\Gamma_1}_{x_1}{\Gamma_3}_{x_2}}{\det\JacR\Gamma} &
				\frac{{\Gamma_1}_{x_2}{\Gamma_2}_{x_1}
					-{\Gamma_1}_{x_1}{\Gamma_2}_{x_2}}{\det\JacR\Gamma}
		\vspace{0.5em} \\	
				\frac{\wt{\rho_{y_2}}{\Gamma_3}_{x_1}
					-\wt{\rho_{x_1}}{\Gamma_2}_{x_1}}{\det\JacR\Gamma} &
				\frac{\wt{\rho_{x_2}}{\Gamma_3}_{x_1}
					+\wt{\rho_{x_1}}{\Gamma_1}_{x_1}}{\det\JacR\Gamma} &
				\frac{-\wt{\rho_{y_2}}{\Gamma_1}_{x_1}
					-\wt{\rho_{x_2}}{\Gamma_2}_{x_1}}{\det\JacR\Gamma} 
			\end{array} \right),
	\eeas
wherever $\JacR\Gamma$ is invertible. In particular, $\JacR\Gamma(0)=(\JacR\Gamma)^{-1}(0)=\id$ We may, therefore, locally invert $\Gamma$ (as a $\cont^3$-smooth function) in some neighborhood $W_1\subset U'_1$ of $0$. Let 
	\bes
		(X_1,Y_1,X_2)=\Gamma^{-1}(x_1,y_1,x_2).
	\ees
$\Gamma$ is constructed to `straighten' $v$ --- i.e., $\JacR\Gamma(\smpartl{}{Y_1})=v$. So, if we view $X_1$ and $X_2$ as $\cont^3$-smooth functions on $W:=\Gamma(W_1)\cap U'$, they have linearly independent differentials and $v(X_1)\equiv v(X_2)\equiv 0$. Thus, $dX_1\wedge dX_2\neq 0$ everywhere on $W$ and $dX_1(v)\equiv dX_2(v)\equiv \theta_\rho(v)\equiv 0$ on $W$. So, it must be the case that
	\bes
		\theta_\rho(\cdot)=\omega_1(\cdot)dX_1(\cdot)+\omega_2(\cdot)dX_2(\cdot),
	\ees
for some $\omega_1,\omega_2\in\cont^2(W)$. Substituting the expressions for $\theta_\rho$, $dX_1$ and $dX_2$ (the latter two can be read off the matrix $(\JacR\Gamma)^{-1}$ above), we get $\omega_1=$
	\bes
	\frac{-{\Gamma_1}_{x_1}\wt{\rho_{y_2}}(\rho_{y_1}\rho_{y_2}+\rho_{x_1}\rho_{x_2})
			+{\Gamma_2}_{x_1}(\wt{\rho_{x_1}}(\rho_{y_2}^2+\rho_{x_2}^2)
			-\wt{\rho_{x_2}}(\rho_{y_1}\rho_{y_2}+\rho_{x_1}\rho_{x_2}))
			-{\Gamma_3}_{x_1}\wt{\rho_{y_2}}(\rho_{x_2}^2+\rho_{y_2}^2)}
			{\rho_{y_2}\wt{\rho_{y_2}}}
	\ees
and $\omega_2=$
	\bes
	\frac{-{\Gamma_1}_{x_2}\wt{\rho_{y_2}}(\rho_{y_1}\rho_{y_2}+\rho_{x_1}\rho_{x_2})
			+{\Gamma_2}_{x_2}(\wt{\rho_{x_1}}(\rho_{y_2}^2+\rho_{x_2}^2)
			-\wt{\rho_{x_2}}(\rho_{y_1}\rho_{y_2}+\rho_{x_1}\rho_{x_2}))
			-{\Gamma_3}_{x_2}\wt{\rho_{y_2}}(\rho_{x_2}^2+\rho_{y_2}^2)}
			{\rho_{y_2}\wt{\rho_{y_2}}},
	\ees
where, once again, $\wt{f}:=f\circ\Gamma$. Observe that $\omega_1(0)=0$ and $\omega_2(0)=1$. Thus, for some neighborhood, $V\subset W$, of the origin, $\omega_2\neq 0$ and 
	\bes
		\theta_\rho=\omega_2(Y_1dX_1+dX_2),
	\ees
where $Y_1:=\omega_1/\omega_2$. Finally, set
	\bes
		\alpha:=\frac{1}{\omega_2},\ \pi_1:=X_1,\ \pi_2:=-\frac{Y_1}{4\lam}\ \text{and}\ \pi_3:=X_2+\frac{X_1Y_1}{2}.
	\ees
Then, on $V$, 
	\be\label{eq_contact}
		\alpha\theta_\rho=-2\lam\pi_2d\pi_1+2\lam\pi_1d\pi_2+d\pi_3=\Pi^*(\theta_{\rho^\lam})
	\ee
and $\alpha(0)=1$.

Refering to \eqref{eq_hdergam} and the formulae for $\omega_1$, $\omega_2$ and $(\JacR\Gamma)^{-1}$, we get   
	\bea\label{eq_contder}
		\JacR\Pi(0)=
			\left( \begin{array}{ccc}
				1 & 0	& 0 \vspace{0.25em} \\
							0 &1 & -\frac{\ima\mu}{2\lam} \vspace{0.25em} \\
							0 & 0 & 1 
			\end{array} \right).
	\eea
We have, thus, constructed the required map. 
\end{proof}

%----------
%---- Contact Transf Lemma Estimate----
%----------
We now show that the contact transformation constructed above satisfies an estimate crucial to our analysis. 

\begin{lemma}\label{lem_darbest} Let $\Pi$ and $V$ be as in the proof of Lemma \ref{lem_darb} and $\V\Subset V$ be a neighborhood of the origin. Then, there is an $ e_1\in\cont(\V)$ with $e_1(0)=0$ and a $ c_1>0$ such that, for all $w'\in\V$ and $z'\in\rl^3$,  
	\bea\label{eq_darbest}
		&&|(z'-w')\tran\cdot \Hess_\rl\pi_3(w')\cdot(z'-w')|\notag \\
		&&\qquad \qquad\leq e_1(w')|z'-w'|^2
		+ c_1(|z_1-w_1||x_2-u_2|+|x_2-u_2|^2).
	\eea	
\end{lemma}
\begin{proof}
Recall that $\pi_3=X_2+\frac{X_1Y_1}{2}$. We refer to the construction in the proof of Lemma \ref{lem_darb} and collect the following data:
	\beas
	\label{eq_derx1}	
		&&(X_1)_{x_1}(0)=1,\ (X_1)_{y_1}(0)=0; \\
	\label{eq_dery1}
		&&(Y_1)_{x_1}(0)=0,\ (Y_1)_{y_1}(0)=-4\lam;\\
%	\label{eq_derx2I}
%		&&(X_2)_{x_1}(0)=0,\ (X_2)_{y_1}(0)=0,\ (X_2)_{x_2}(0)=1;\\
	\label{eq_derx2}
		&&(X_2)_{x_1x_1}(0)=0,\ (X_2)_{x_1y_1}(0)=2\lam=(X_2)_{y_1x_1}(0),\ 
			(X_2)_{y_2y_2}(0)=0.
	\eeas
Next, we write out the relevant terms.
	\beas
		&&(z'-w')\tran\cdot \Hess_\rl\pi_3(w')\cdot(z'-w')\\
		&&=\Big({X_2}_{x_1x_1}(w')
			+{X_1}_{x_1}(w'){Y_1}_{x_1}(w')
			+\frac{1}{2} Y_1(w'){X_1}_{x_1x_1}(w')
			+\frac{1}{2} X_1(w'){Y_1}_{x_1x_1}(w')\Big)(x_1-u_1)^2\\
		 &&\quad +\Big(2{X_2}_{x_1y_1}(w')+{X_1}_{x_1}(w'){Y_1}_{y_1}(w')
				+{X_1}_{y_1}(w'){Y_1}_{x_1}(w')\Big)(x_1-u_1)(y_1-v_1)\\
		&&\quad +\Big(Y_1(w'){X_1}_{x_1y_1}(w')
			+ X_1(w'){Y_1}_{x_1y_1}(w')\Big)(x_1-u_1)(y_1-v_1)\\
		&& \quad +\Big({X_2}_{y_1y_1}(w')
			+{X_1}_{y_1}(w'){Y_1}_{y_1}(w')
			+\frac{1}{2} Y_1(w'){X_1}_{y_1y_1}(w')
			+\frac{1}{2} X_1(w'){Y_1}_{y_1y_1}(w')\Big)(y_1-v_1)^2\\
		&& \quad +2{\pi_3}_{x_1x_2}(w')(x_1-u_1)(x_2-u_2)
					+2{\pi_3}_{y_1x_2}(w')(y_1-v_1)(x_2-u_2)
					+{\pi_3}_{x_2x_2}(w')(x_2-u_2)^2.
\eeas
Now, the coefficients of $(x_1-u_1)^2$, $(x_1-u_1)(y_1-v_1)$ and $(y_1-u_1)^2$ in the above expansion all vanish at the origin (see data listed above). Thus, we obtain \eqref{eq_darbest}.
\end{proof}

%----------
%---- Contact Transf Lemma Normal Direction----
%----------
All that remains is to extend the above transformation to $\Om$. For this, let $V$ be as in Lemma \ref{lem_darb} and $G_\rho:V\times\rl\rightarrow\C^2$ be the map
	\bes
		(x_1,y_1,x_2,y_2)\mapsto(x_1,y_1,x_2,F_\rho(x_1,y_1,x_2)+y_2).
	\ees
$G_\rho$ is evidently a $\cont^4$-smooth diffeomorphism with $G(V\times(0,t])\subset\Om$ for some $t>0$.  We note the following facts about $G_\rho$:
\begin{itemize}
\item $\JacR G_\rho(0)=\id$
\item $(G_\rho)^*(d\rho)=\left(\partl{\rho}{y_2}\circ G_\rho\right) dy_2$ and $(G_\rho)^*\left(\frac{\partial\rho-\conj\partial\rho}{i}\right)
		=\theta_\rho$ on $V\times\{0\}$.
\end{itemize}

\begin{lemma}\label{lem_darbext} There is a neighborhood $ U\subset\CC$ of the origin and a $\cont^2$-diffeomorphism $\Psi: U \rightarrow \Psi(U)\subset \CC$ such that
	\begin{itemize}
		\item $\Psi(0)=0$, $\Psi(\Om\cap U)=\mathcal{S}_\lam\cap\Psi( U)$ and 
				$\Psi(\bdy\Om\cap U)=\bdy \mathcal{S}_\lam\cap\Psi( U)$;
		\item $\det\JacR\Psi(0)=1$ and $\det\JacR\Psi^{\operatorname{res}}(0)=1$; and
		\item if $\mathfrak{l}_\rho$ and $\mathfrak{l}_\lam$ denote the Cauchy-Leray map of $\rho$ and $\rho^\lam$, 
				respectively, then 
				\bea\label{eq_darbestmain}
				%\begin{align}
				&& \qquad\qquad \big|\mathfrak{l}_\rho(z,w)-\mathfrak{l}_\lam(\Psi(z),\Psi(w))\big| \\
			 &&\leq { ( e(w)+ D(z-w))
				\left(|\mathfrak{l}_\lam(\Psi(z),\Psi(w))+||z-w||^2\right)
				 + c|\mathfrak{l}_\lam(\Psi(z),\Psi(w))|^2},\notag 
				%\end{align}
				\eea
				on $\{(z,w)\in\conj\Om\times U:||z-w||\leq\tau\}$, for some choice
				 of $ e\in\cont(U)$ with $e(0)=0$, 
				$ D(\zt)=o(1)$ as $|\zt|\rightarrow 0$, and constants
				 $ c,\tau>0$. 
	\end{itemize}
\end{lemma} 
\begin{proof} Let $\Psi=(\Psi_1,\Psi_2):=G_{\rho^\lam}\circ(\Pi,\id)\circ G_\rho^{-1}$, where $\id$ is the identity map on $\rl$, and $ U\Subset G_\rho(V\times[-t,t])$ is a neighborhood of the origin. We use the notation $(\Psi_1,\Psi_2)=(\psi_1+i\psi_2,\psi_3+i\psi_4)$. The regularity and mapping properties of $\Psi$ follow from its definition We also clarify that $\partial\rho^\lam(\Psi(w))$ denotes $\partial\rho^\lam$ evaluated at $\Psi(w)$. Since $\id^*(-dy_2)=-dy_2$ and  $\Pi^*(\theta_{\rho^\lam})=\alpha\theta_\rho$ on $\{y_2=0\}$, 
	\bes 
		\Psi^*(d\rho^\lam)
		=\alpha_1(d\rho)
	\ees
and
	\bes
		\Psi^*\left(\frac{\partial\rho^\lam-\conj\partial\rho^\lam}{i}\right)
		=\alpha_2\left(\frac{\partial\rho-\conj\partial\rho}{i}\right),
	\ees
on $\bdy\Om\cap U$, where $\alpha_1(x_1,y_1,x_2,y_2)=-1/\left(\partl{\rho}{y_2}(G_\rho(x_1,y_1,x_2,y_2))\right)$ and $\alpha_2(x_1,y_1,x_2,y_2)=\alpha(x_1,y_1,x_2)$.
Therefore, for all $w\in\bdy\Om\cap U$ and $z\in\CC$,
	\begin{flalign}
		& 2\lbr\partial\rho^\lam(\Psi(w)),\JacR\Psi(w)(z-w)\rbr \label{eq_psiest1} \\
		& \quad =2\rea\lbr\partial\rho^\lam(\Psi(w)),\JacR\Psi(w)(z-w)\rbr
		+2i\ima\lbr\partial\rho^\lam(\Psi(w)),\JacR\Psi(w)(z-w)\rbr\notag \\
		&\quad=\ldbr(\partial\rho^\lam+\conj\partial\rho^\lam)
																	(\Psi(w)),\JacR\Psi(w)(z-w)\rdbr
		+i\ldbr\frac{\partial\rho^\lam-\conj\partial\rho^\lam}{i}
							(\Psi(w)),\JacR\Psi(w)(z-w)\rdbr \notag \\
		&\quad =\ldbr\Psi^*(\partial\rho^\lam+\conj\partial\rho^\lam)(w),z-w\rdbr
		+i\ldbr\Psi^*\left(\frac{\partial\rho^\lam-\conj\partial\rho^\lam}{i}\right)
																				(w),z-w\rdbr \notag\\
		&\quad =\alpha_1(w)\ldbr(\partial\rho+\conj\partial\rho)(w),z-w\rdbr
		+i \alpha_2(w)\ldbr\left(\frac{\partial\rho-\conj\partial\rho}{i}\right)
																				(w),z-w\rdbr\notag \\
		&\quad =2\alpha_1(w)\rea\lbr\partial\rho(w),z-w\rbr
			+2i \alpha_2(w)\ima\lbr\partial\rho(w),z-w\rbr. \notag
	\end{flalign}

Now, since $\rho^\lam:=\lam|z_1|^2-y_2$, $\partl{\rho^\lam}{z_1}(\Psi(z))=\lam\conj{\Psi_1(z)}$ and $\partl{\rho^\lam}{z_2}(\Psi(z))=\dfrac{i}{2}$. Therefore, there is a $\tau_1>0$ such that on $\{(z,w)\in\rl^4\times U:||z-w||\leq\tau_1\}$,
	\bea\label{eq_psiest2}
		&&\left|\lbr\partial\rho^\lam(\Psi(w)),\Psi(z)-\Psi(w)-\JacR\Psi(w)(z-w)\rbr\right|\notag \\
		&&\qquad \qquad \leq c|\Psi_1(w)|\cdot||z-w||^2+\frac{1}{2}R_1(z-w)+R_2(z-w), 
	\eea
where, $c'>0$, $R_1(z-w)=|(z-w)^{\tran}\cdot(\Hess_\rl\psi_3(w)
+\Hess_\rl\psi_4(w))\cdot(z-w)|$, and $R_2(\zt)=o(|\zt|^2)$ as $|\zt|\rightarrow 0$. Observe that $\psi_3(z',y_2)=\pi_3(z')$ and $\psi_4(z',y_2)=\pi_1(z')^2+\pi_2(z')^2+y_2-F(z')$. As,  
	\beas
		{\psi_4}_{x_1x_1}(w)&=&2\sum_{j=1}^2
		({\pi_j}_{x_1}(w')^2+\pi_j(w'){\pi_j}_{x_1x_1}(w'))-F_{x_1x_1}(w'),\\
		{\psi_4}_{y_1y_1}(w)&=&2\sum_{j=1}^2
		({\pi_j}_{y_1}(w')^2+\pi_j(w'){\pi_j}_{y_1y_1}(w'))-F_{y_1y_1}(w')\ \text{and}\\
		{\psi_4}_{x_1y_1}(w)&=&2\sum_{j=1}^2
		({\pi_j}_{x_1}(w'){\pi_j}_{y_1}(w')+\pi_j(w'){\pi_j}_{x_1y_1}(w'))
		-F_{x_1y_1}(w')
	\eeas 
all vanish at $w=0$, we have, for all $(z,w)\in\rl^4\times U$,
	\bea\label{eq_psi4est}
	&& |(z-w)^{\tran}\cdot\Hess_\rl\psi_4(w)\cdot(z-w)|\notag \\
		&&\qquad\qquad \leq  e_2(w)||z-w||^2+ c_2(|z_1-w_1||z_2-w_2|+|z_2-w_2|^2),
	\eea
where $ e_1\in\cont( U)$ with $e_1(0)=0$, and $ c_1>0$ is a constant. Combining \eqref{eq_psiest2}, \eqref{eq_darbest} and \eqref{eq_psi4est} (and adding $c'|\Psi_1|$, $ e_1$ and $ e_2$), we have that 
	\bea\label{eq_psiest3}
		&&\quad \lefteqn{ \mathcal{A}
		:=\left|\lbr\partial\rho^\lam(\Psi(w)),\Psi(z)-\Psi(w)-\JacR\Psi(w)(z-w)\rbr
		\right|}\\
		&& \leq( e_3(w)+ D_3(z-w))||z-w||^2+ c_3(|z_1-w_1||z_2-w_2|+|z_2-w_2|^2), \notag
	\eea
on $\{(z,w)\in\rl^4\times U:||z-w||\leq\tau_3\}$, for some $ e_3\in\cont( U)$ with $e_3(0)=0$, $ D_3(\zt)=o(1)$ as $|\zt|\rightarrow 0$, and constants $ c_3,\tau_3>0$. 

Next, we have that 
	\bea
		|\Psi_2(z)-\Psi_2(w)|
		&=&2\left|\lbr\partial\rho^\lam(\Psi(w)),\Psi(z)-\Psi(w)\rbr
		-\conj{\Psi_1(z)}(\Psi_1(z)-\Psi_1(w))\right|\notag \\
		&\leq&
		 c_4\left|\lbr\partial\rho^\lam(\Psi(w)),\Psi(z)-\Psi(w)\rbr\right|
		+ e_4(w)||z-w||, \label{eq_psi2est}
	\eea
on $\{(z,w)\in\rl^4\times U:||z-w||\leq\tau_4\}$, for some choice of $ e_4$, $ c_4$ and $\tau_4$ as before. Also, if $\Psi^{-1}=(\hat\psi_1,\hat\psi_2,\hat\psi_3,\hat\psi_4)$, then $\JacR\hat\psi_3(0)=(0,0,1,0)$ and $\JacR\hat\psi_4(0)=(0,0,0,1)$. So, we are permitted to conclude that
	\be\label{eq_z2w2est}
		|z_2-w_2|\leq  c_4|\Psi_2(z)-\Psi_2(w)|+( e_5(w)+ D_5(z-w))||z-w||,
	\ee
on $\{(z,w)\in\rl^4\times U:||z-w||\leq\tau_5\}$, for some $ e_5$, $ c_5$, $ D_5$ and $\tau_5$ as before.
 
Finally, as $\alpha_1(0)=\alpha_2(0)=1$, \eqref{eq_psiest1}, \eqref{eq_psiest3}, \eqref{eq_psi2est} and \eqref{eq_z2w2est} combine to give $e$, $c$, $D$ and $\tau$ with the required properties, such that
	\begin{flalign}
		&|\mathfrak{l}_\rho(z,w)-\mathfrak{l}_\lam(\Psi(z),\Psi(w))|\notag \\
		&\leq\left|\lbr\partial\rho(w),z-w\rbr
				-\lbr\partial\rho^\lam(\Psi(w)),\JacR\Psi(w)(z-w)\rbr\right|+\mathcal{A},\notag
					\\ %\text{(see \eqref{eq_psiest3})}	
		&\leq( e(w)+ D(z-w))\left(|\mathfrak{l}_\lam(\Psi(z),\Psi(w))|+||z-w||^2\right)
				+ c|\mathfrak{l}_\lam(\Psi(z),\Psi(w))|^2\notag
	\end{flalign}
on $\{(z,w)\in\rl^4\times U:||z-w||\leq\tau\}$.
\end{proof}

%-------------
%-------------
%---Cvxfn-----
%-------------
%-------------
\noindent{\bf Convexification.} Now, we return to general strongly pseudoconvex domains. Assume $0\in\bdy\Om$ and the outward unit normal vector to $\bdy\Om$ at $0$ is $(0,-i)$. Let $\rho$ be a $\cont^2$-smooth strictly plurisubharmonic defining function of $\Om$ such that $||\nabla\rho(0)||=1$. Now, $\rho$ has the following second-order Taylor expansion about the origin:
	\bes
		\rho(w)=\ima\left(-w_2
		+i\sum_{j,k=1}^{2}\secpartl{\rho(0)}{z_j}{z_k}w_jw_k\right)
		+\sum_{j,k=1}^{2}\secpartl{\rho(0)}{z_j}{\conj{z_k}}w_j\conj{w_k}
		+o(|w|^2).
	\ees
Using a classical trick, attributed to Narasimhan, we convexify $\Om$ near the origin via the map $\Phi$ given by:
	\beas
		w_1&\mapsto& \Phi_1(w)=w_1\\
		w_2&\mapsto& \Phi_2(w)=w_2
		-i\sum_{j,k=1}^{2}\secpartl{\rho(0)}{z_j}{z_k}w_jw_k.
	\eeas
Owing to the inverse function theorem, $\Phi$ is a local biholomorphism on some neighborhood $U$ of $0$. We may further shrink $U$ so that the strong convexity of $\Phi(\bdy\Om\cap U)$ at $0$ propagates to all of $\Phi(\bdy\Om\cap U)$. We collect the following key observations:
	\begin{itemize}
		\item $\JacR\Phi(0)=\id$;
		\item If $\hat\rho:=\rho\circ\Phi^{-1}$, then $	\hat\rho(w)=-\ima w_2+
				\sum\limits_{j,k=1}^{2}\secpartl{\rho(0)}{z_j}{\conj{z_k}}w_j\conj{w_k}
				+o(|w|^2).$
		\item Let $\mathfrak{p}$ denote the Levi-polynomial of $\rho$, 
				$\mathfrak{l}_{\hat\rho}(z,w)$ be the Cauchy-Leray map of $\hat\rho$, and $\partial\hat\rho(\Phi(w))$ denote $\partial\hat\rho$ evaluated at $\Phi(w)$. Then, for any 
				neighborhood $U_1\Subset U$ of the origin, there is a $\tau>0$ 
				such that, on $\{(z,w)\in\CC\times U_1:||z-w||\leq\tau\}$,
			\begin{flalign}\label{eq_LeviLer}
				&|\mathfrak{p}(z,w)-\mathfrak{l}_{\hat\rho}(\Phi(z),\Phi(w))| 
				=|\mathfrak{p}(z,w)-\mathfrak{l}_{\hat\rho}(\Phi(z),\Phi(w))| 
				 \\
				&\leq  \left|\lbr\partial\rho(w),z-w\rbr
				-\lbr\partial\hat\rho(\Phi(w)),\JacR\Phi(w)(z-w)\rbr\right|\notag \\
				&\qquad
 				+\frac{1}{2}\left|\sum_{j,k=1}^{2}\left(\secpartl{\rho(w)}{z_j}{z_k}
				+2i\partl{\hat\rho(\Phi(w))}{w_2}\secpartl{\rho(0)}{z_j}{z_k}
				\right)(z_j-w_j)(z_k-w_k)\right|\notag\\
				& \leq \left|\lbr\partial\rho(w),z-w\rbr
				-\lbr\Phi^*(\partial\hat\rho)(w),z-w\rbr\right|\notag\\
				&\qquad
				+\frac{1}{2}\left|\sum_{j,k=1}^{2}\left(\secpartl{\rho(0)}{z_j}{z_k}
				+o(1)	+\left(-1+o(|w|)\right)
				\secpartl{\rho(0)}{z_j}{z_k}\right)(z_j-w_j)(z_k-w_k)\right|\notag\\
				& \leq e(w)||z-w||^2,\notag
			\end{flalign}
for some $e\in\cont(U)$ with $e(0)=0$.
\end{itemize}

%-------------
%-------------
%---Local Estimates---
%-------------
%-------------
\noindent{\bf Main Local Estimate.} We combine the maps constructed above:
\begin{lemma}\label{lem_locests} Fix an $\eps>0$. Let $\Om\subset\CC$ be a $\cont^4$-smooth strongly pseudoconvex domain and $\rho$ a strictly plurisubharmonic defining function of $\Om$. Assume that $0\in\bdy\Om$, $\nabla\rho(0)=(0,0,0,-1)$ and $M(\rho)(0)=\lam$. Then, there exists a neighborhood $U$ of the origin, a $\cont^2$-smooth origin-preserving diffeomorphism $\Theta$ on $U$ that carries $\conj\Om\cap U$ onto $\conj {\mathcal{S}_\lam}\cap \Theta(U)$, and a constant $\tau>0$ such that 
	\begin{itemize}
	\item $1-\eps\leq\dfrac{\vol(\Theta(V))}{\vol(V)}\leq \dfrac{1}{1-\eps}$, for every 
			Jordan measurable $V\subset U$; 
	\item $1-\eps\leq\dfrac{\rvol(\Theta(J)')}{\rvol(J')}\leq\dfrac{1}{1-\eps}$, for every Jordan 
			measurable $J\subset \bdy\Om\cap U$; and  
	\item if $\mathfrak{p}$ is the Levi polynomial of $\rho$ and $\mathfrak{l}_\lam$ is the Cauchy-Leray map of 
			$\rho^\lam$, then
			\bes
			 |\mathfrak{p}(z,w)-\mathfrak{l}_\lam(\Theta(z),\Theta(w))|
			\leq \eps(|\mathfrak{p}(z,w)|+|\mathfrak{l}_\lam(\Theta(z),\Theta(w))|)
			\ees
			 on $(U\times U)\cap\Om_\tau$. 
	\end{itemize}
\end{lemma}
\begin{proof} The needed map is $\Psi\circ\Phi$ (from Lemma \ref{lem_darbext} and the convexification procedure above). The mapping and volume distortion properties follow from those of $\Psi$ and $\Phi$. The estimate is a combination of \eqref{eq_LeviLer}, \eqref{eq_darbestmain} and \eqref{eq_Leviest}.
\end{proof}

The following lemma is an application of Lemma \ref{lem_main} and gives us a local version of our main theorem. 
%----------
%---- Final Estimates ----
%----------
\begin{lemma}\label{lem_locfin} Let $\Om$, $f$ and $\rho$ be as in Theorem \ref{thm_MAIN}. Fix an $\eps\in(0,1/3)$ and a point $q\in\bdy\Om$. Then, there exists a neighborhood $U_{q,\eps}$ of $q$ such that for every Jordan measurable pair $J, H\subset \bdy\Om\cap U_{q,\eps}$ such that $J\subset\inte_{\bdy\Om} H$,
	\beas
	  (1-\eps)^{31}\ \lkor
		 \frac{\lam(q)^{\frac{1}{2}}
			 s(J)^{\frac{3}{2}}}{\sqrt{n}}
	 \leq 
		 v(\Om;\polcl{J\subset H;f}{n})
		\leq (1-\eps)^{-19}\  \lkor\frac{\lam(q)^{\frac{1}{2}} s(J)^{\frac{3}{2}}}{\sqrt{n}}
	\eeas
for sufficiently large $n$, where $\lam(q):=\dfrac{4M(\rho)(q)}{||\nabla\rho(q)||^3}$ and $s$ is the Euclidean surface area measure on $\bdy\Om$.  
\end{lemma}
\begin{proof} First, we set $\hat\eps=c\eps$, where $c<1$ will be revealed later. Let $\rho$ be the strictly plurisubharmonic defining function of $\Om$ for which \eqref{eq_buli} in Theorem \ref{thm_MAIN} holds. Let $A:\CC\rightarrow\CC$ be a holomorphic isometry that takes $q$ to the origin and the outer unit normal at $q$ to $(0,-i||\nabla\rho(q)||)$. Set $\hat\rho(z):=||\nabla\rho(q)||^{-1}\rho(A^{-1}z)$. Then, $A(\Om)$ and 
$\hat\rho$ satisfies the hypotheses of Lemma \ref{lem_locests}, with $M(\hat\rho)(0)=\lam(q)$. Suppose $\Theta$, $U$ and $\tau$ are the map, neighborhood and constant, respectively, granted by Lemma \ref{lem_locests}, and $\hat\levip$ is the Levi polynomial of $\hat\rho$. Then,
		\be\label{eq_conslem}
			 |\hat\levip(z,w)-\mathfrak{l}_{\lam(q)}(\Theta(z),\Theta(w))|
			\leq \hat\eps(|\hat\levip(z,w)|+|\mathfrak{l}_{\lam(q)}(\Theta(z),\Theta(w))|)
			\ee
			 on $(U\times U)\cap A(\Om)_\tau$. Also note that
	\be\label{eq_modlevi} 
		||\nabla\rho(q)||\hat\levip(Az,Aw)=\levip(z,w).
	\ee 

Next, set $U_q:=A^{-1}(U)$ and $\Theta_q:=\Theta\circ A$. Note that $\Theta_q$ maps $\conj \Om$ to $\conj {\mathcal{S}}_{\lam(q)}$ locally near $q$. We define 
	\bea
		\tilde f(z,w)&:=&\frac{f(z,w)}{||\nabla\rho(q)||};\label{eq_tildf}\\
		g(z,w)&:=&f_{\mathcal{S}_{\lam(q)}}\big(\Theta_qz,\Theta_qw\big);\ \text{and}\label{eq_g}\\
		\tilde g(z,w)&:=& a(w,w)\cler_{\lam(q)} \big(\Theta_qz,\Theta_qw\big)=a(w,w)\frac{i}{2\lam(q)}f_{\mathcal{S}_{\lam(q)}}\big(\Theta_qz,\Theta_qw\big).\label{eq_tildg}
	\eea
So, when defined,
	\bea
		C(w,\de;\tilde f)&=&C\left(w, ||\nabla\rho(q)||\de;f\right);\ \text{and} \label{eq_fcuts}\\
		C(w,\de;\tilde g)&=&C\left(w,\frac{2\lam(q)}{|a(w,w)|}\de;g\right).
		\label{eq_gcuts}
	\eea
Thus, for our point of interest, there is little difference between $f$ and $\tilde f$ (and, between $g$ and $\tilde g$). Keeping this observation in mind, we will apply Lemma \ref{lem_main} to $\tilde f,\tilde g\in\cont(\conj\Om\times (\bdy\Om\cap U_q))$ (see Remark \ref{rmk_locglob}). To bound $|\tilde f(z,w)-\tilde g(z,w)|$ from above, we estimate $|\tilde f(z,w)-a(z,w)\hat\levip(Az,Aw)|$, $|a(z,w)\hat\levip(Az,Aw)-a(z,w)\cler_{\lam(q)} \big(\Theta_qz,\Theta_qw\big)|$ and $|a(z,w)\cler_{\lam(q)} \big(\Theta_qz,\Theta_qw\big)-\tilde g(z,w)|$.
  
By \eqref{eq_buli}, we can find a $\tau_1\in(0,\tau]$ such that 
	\be\label{eq_fandlevi}
		|\tilde f(z,w)-a(z,w)\hat\levip(Az,Aw)|
		=\frac{|f(z,w)-a(z,w)\levip(z,w)|}{||\nabla\rho(q)||}\leq  \frac{\hat\eps}{||\nabla\rho(q)||}|\levip(z,w)|\ \ \ \text{on}\ \Om_{\tau_1}.
	\ee
By Lemma \ref{lem_locests}, \eqref{eq_modlevi}, \eqref{eq_tildg} and the continuity of $a$ on $\conj\Om_\tau$, we shrink $\tau_1$ so that on $(U_q\times U_q)\cap\Om_{\tau_1}$,
	\bea\label{eq_leviandler}
		&&|a(z,w)\hat\levip(Az,Aw)-a(z,w)\cler_{\lam(q)} \big(\Theta_qz,\Theta_qw\big)|\notag \\	
		&&\leq|a(z,w)||\hat\levip(Az,Aw)-\cler_{\lam(q)} \big(\Theta_qz,\Theta_qw\big)| \notag \\
		&&\leq \hat\eps|a(z,w)|\left(|\hat\levip(Az,Aw)|+|\cler_{\lam(q)} \big(\Theta_qz,\Theta_qw\big)|\right)\notag\\
		&&= \hat\eps|a(z,w)|\left(\frac{|\levip(z,w)|}{||\nabla\rho(q)||}+\frac{|\tilde g(z,w)|}{|a(w,w)|}\right)\notag\\
		&&\leq\hat\eps \left(\frac{\max_{\Om_\tau}|a(z,w)|}{||\nabla\rho(q)||}\right)|\levip(z,w)|
			+\hat\eps \left(\frac{\max_{\Om_\tau}|a(z,w)|}{\min_{\bdy\Om}|a(w,w)|}\right)|\tilde g(z,w)|,
	\eea
and
	\bea\label{eq_lerandg}
		|a(z,w)\cler_{\lam(q)} \big(\Theta_qz,\Theta_qw\big)-\tilde g(z,w)|&=&|a(z,w)-a(w,w)|
				\cdot|\cler_{\lam(q)} \big(\Theta_qz,\Theta_qw\big)|\notag \\
		&\leq& \frac{\hat\eps}{\min_{\bdy\Om}|a(w,w)|}|\tilde g(z,w)|.
	\eea
Lastly, by \eqref{eq_buli}, there exist $\tau_2\in(0,\tau_1]$ and $l>0$ such that 
	\be\label{eq_leviandf}
		|\mathfrak{p}(z,w)|\leq l|\tilde f(z,w)|\ \ \ \text{on}\ \Om_{\tau_2}.
	\ee
Now, set 
	\bes
		c=\frac{1}{2}\min\left\{1,\frac{||\nabla\rho(q)||}{l},
								\frac{||\nabla\rho(q)||}{l\max_{\Om_\tau}|a(z,w)|},
								\frac{\min_{\bdy\Om}|a(w,w)|}{\max_{\Om_\tau}|a(z,w)|},
								\min_{\bdy\Om}|a(w,w)|\right\}.
	\ees 
Then, adding \eqref{eq_fandlevi}, \eqref{eq_leviandler} and \eqref{eq_lerandg}, and using \eqref{eq_leviandf}, we get
	\bes
		|\tilde f(z,w)-\tilde g(z,w)|\leq \eps\left(|\tilde f(z,w)|+|\tilde g(z,w)|\right)\qquad \text{on}\ (U_q\times U_q)\cap\Om_{\tau_2}.
	\ees

We now need to show that $\tilde g$ satisfies the remaining hypotheses of Lemma \ref{lem_main}. But these are conditions on the cuts of $\tilde g$, which are identical to the cuts of $g$ (by \eqref{eq_gcuts}). So, we work with $g$ instead. Let $U_{q,\eps}\Subset U_q$ be an open neighborhood of $q$, and $\de_0>0$ be such that $C(w,\de;g)\subset V_q$ for all $w\in U_{q,\eps}\cap\bdy\Om$ and $\de<\de_0$. Then, there is a diffeomorphism
	\be\label{eq_cutmap}
		\Theta_q=\Theta\circ A:C(w,\de;g)\rightarrow C\big(\Theta_qw,\de;f_{\mathcal{S}_{\lam(q)}}\big),
	\ee
for  $w\in U_{q,\eps}\cap\bdy\Om$ and $\de< \de_0$. Therefore, exploiting Lemma \ref{lem_adm}, we get
	\begin{enumerate}
		\item  $C(w,\de;g)$ is Jordan measurable for all  $w\in U_{q,\eps}\cap\bdy\Om$ and $\de<\de_0$; 
		\item If $w^1,...,w^m\in  U_{q,\eps}\cap\bdy\Om$, $m\in\pnat$, then
			\beas
				\vol\left(\bigcup_{j=1}^mC(w^j,(1+t)\de;g)\right)
					&\leq& \frac{1}{1-\eps}\vol\left(\bigcup_{j=1}^mC(\Theta_qw^j,(1+t)\de;f_{\mathcal{S}_{\lam(q)}})\right)\\
				&\leq& \frac{(1+t)^3}{1-\eps}\vol\left(\bigcup_{j=1}^mC(\Theta_qw^j,\de;f_{\mathcal{S}_{\lam(q)}})\right)\\	
				&\leq &\dfrac{(1+t)^3}{(1-\eps)^2}\vol\left(\bigcup_{j=1}^mC(w^j,\de;g)\right),
			\eeas
for all $t\in(0,16)$ and $\de_j\leq \de_0/16$, $j=1,...,m$. Thus, $g$ satisfies the doubling property \eqref{eq_doub} with
quantifiers $\de_{g}=\de_0/16$ and $D(t)=(1-\eps)^{-2}(1+t)^3$. 
\end{enumerate}

Lastly, we further shrink $U_{q,\eps}$ --- if necessary --- to ensure that for any $s$-measurable set $J\subset(U_{q,\eps}\cap\bdy\Om)$,	
		\be\label{eq_dist}
			1-\eps\leq \frac{s(J)}{\rvol(J'')}\leq \dfrac{1}{1-\eps},	
		\ee
where $J''$ is the orthogonal projection of $J$ onto $T_q\bdy\Om$, and by $\rvol(J'')$, we really mean $\rvol(A(J''))$.

We are now ready to estimate. Consider Jordan measurable compact sets $J\subset H\subset( U_{q,\eps}\cap\bdy\Om)$ such that $J\subset\inte_{\bdy\Om} H$. By \eqref{eq_fcuts}, \eqref{eq_f<=g} from Lemma \ref{lem_main}, \eqref{eq_gcuts}, the volume-distortion properties of $\Theta_q$ --- see Lemma \ref{lem_locests} and recall that $A$ is an isometry --- Corollary \ref{cor_levi}  and \eqref{eq_dist}, we have that
	\beas
		&&\limsup_{n\rightarrow \infty}\sqrt{n}\ v(\Om;\polcl{J\subset H;f}{n})
		=\limsup_{n\rightarrow \infty}\sqrt{n}\ v(\Om;\polcl{J\subset H;\tilde f}{n})\\
		&&\leq \dfrac{1}{(1-\eps)^2}\left(1+\frac{(1+\eps)^2}{(1-\eps)^2}-1\right)^3
		\limsup_{n\rightarrow \infty}\sqrt{n}\ v(\Om;\polcl{J\subset H;\tilde g}{n})\\
		&&= \dfrac{1}{(1-\eps)^2}\left(1+\frac{(1+\eps)^2}{(1-\eps)^2}-1\right)^3
		\limsup_{n\rightarrow \infty}\sqrt{n}\ v(\Om;\polcl{J\subset H; g}{n})\\	
		&&\leq(1-\eps)^{-14}
		\limsup_{n\rightarrow \infty}\sqrt{n}\ (1-\eps)^{-1}
		v(\mathcal{S}_{\lam(q)};\polcl{\Theta_q J\subset\Theta_q H;f_{\mathcal{S}_{\lam(q)}}}{n})\\
		&&\leq(1-\eps)^{-15}\ \lkor\lam(q)^\frac{1}{2}\rvol((\Theta_q J)')^\frac{3}{2}\\
		&&\leq (1-\eps)^{-\frac{33}{2}}\ \lkor\lam(q)^\frac{1}{2}\rvol(J'')^\frac{3}{2}
			\leq(1-\eps)^{-18}\ \lkor\lam(q)^\frac{1}{2}s(J)^\frac{3}{2}.
	\eeas
By a similar argument, but now using $\eqref{eq_f>=g}$ from the statement of Lemma \ref{lem_main}, we get that
	\bes
		\lim_{n\rightarrow \infty}\sqrt{n}\ v(\Om;\polcl{J\subset H;f}{n})
		\geq  (1-\eps)^{30}\ \lkor
\lam(q)^\frac{1}{2}s(J)^\frac{3}{2}.
	\ees
Therefore, for large enough $n$, we get the desired estimates.  
\end{proof}

%%%%%%%%%%%%%%%%%%%%%%%%% 
%%%%%%%%%%%%%%%%%%%%%%%%% 
%%%%%%%%%%%%%%%%%%%%%%%%% Main Proof %%%%%%%%%%%%%%%%%%%%%%%%%%%%
%%%%%%%%%%%%%%%%%%%%%%%%% 
%%%%%%%%%%%%%%%%%%%%%%%%% 
\section{Proof of Theorem \ref{thm_MAIN}}\label{sec_mainproofs}
%----------
%---- TILING ----
%----------
\noindent{\em Proof of Theorem \ref{thm_MAIN}}. Fix an $\eps\in(0,1/3)$. There exists a tiling $\{L_j\}_{1\leq j\leq m}$ of $\bdy\Om$ consisting of Jordan measurable compact sets with non-empty interior such that 
	\begin{itemize}
		\item for each $j=1,...,m$, there is a $q_j\in L_j$ for which $L_j\subset U_{q_j,\eps}$, where the latter comes from Lemma \ref{lem_locfin};
		\item $(1-\eps)\lam(q)\leq \lam(q_j)\leq (1-\eps)^{-1}\lam(q)$, for all $q\in L_j$.
	\end{itemize}
Then, recalling that $\lam(q)=\dfrac{4M(\rho)(q)}{||\nabla\rho(q)||^3}$, we obtain estimates as follows: 
\bea
		4^{-\frac{1}{3}}\int_{\bdy\Om}\sigma_\Om
		 =\int_{\bdy\Om} 4^{\frac{1}{3}}M(\rho)(q)^{\frac{1}{3}}
		\frac{ds(q)}{||\nabla\rho(q)||}
				=\sum_{j=1}^m\int_{L_j} \lam(q_j)^\frac{1}{3}ds(q_j)\notag \\
			\left\{ \begin{aligned}
				 \leq (1-\eps)^{-1}\sum\limits_{j=1}^m 
							\lam(q_j)^{\frac{1}{3}}s(L_j) \\ 
			\geq (1-\eps)\sum\limits_{j=1}^m 
							\lam(q_j)^{\frac{1}{3}}s(L_j). \label{eq_disc}
		\end{aligned} \right.
	\eea
Next, for all $j=1,...,m$, we choose compact Jordan measurable sets $J_j$ and $H_j$ such that $J_j\subset \inte_{\bdy\Om} L_j\subset \inte_{\bdy\Om} H_j\subset U_{q_j,\eps}$ and 
	\be\label{eq_JKest}
		s(J_j)\geq (1-\eps) s(L_j).
	\ee

%----------
%---- UPPER BOUND ----
%----------
$\bm{1.}$ We first estimate $v(\Om;\polcl{f}{n})$ from above. For $j=1,...,m$, choose $P^j\in\polcl{L_j\subset H_j;f}{n_j}$ such that $\vol(\Om\setminus P^j)\leq (1-\eps)^{-1}v(\Om;\polcl{L_j\subset H_j;f}{n_j})$. Let $P$ denote the intersection of all these $P^j$'s. Then, $P$ is an $f$-polyhedron with at most $n_1+\cdots+n_m$ facets. Thus, by Lemma \ref{lem_locfin}, for sufficiently large $n_1,...,n_m$, 
\bea 
	\vol(\Om\setminus P)&\leq&
		 (1-\eps)^{-1}\sum\limits_{j=1}^m v(\Om;\polcl{L_j\subset H_j;f}{n_j})\notag \\
		&\leq& (1-\eps)^{-20}\ \lkor \sum\limits_{j=1}^m
		\frac{\lam(q_j)^{\frac{1}{2}} s(L_j)^{\frac{3}{2}}}{\sqrt{n_j}}\notag \\
		&=& (1-\eps)^{-20}\  \lkor \sum\limits_{j=1}^m
		\lam(q_j)^\frac{1}{3}s(L_j)\left(\frac{\lam(q_j)^\frac{1}{3} s(L_j)}{n_j}\right)^\frac{1}{2}. \label{eq_upfin}
\eea
Now, fix an $n\in\pnat$. Suppose, we set
	\be
		n_j=\left\lfloor\frac{\lam(q_j)^\frac{1}{3} s(L_j)}{\sum_{j=1}^m \lam(q_j)^\frac{1}{3} s(L_j)}n\right\rfloor,\ j=1,...,m.
	\ee
Then,
	\be\label{eq_sum}
		n_1+\cdots+n_m\leq n;
	\ee
and 
	\be \label{eq_bdnj}
		(1-\eps)\frac{\lam(q_j)^\frac{1}{3} s(L_j)}{\sum_{j=1}^m \lam(q_j)^\frac{1}{3} s(L_j)}n\leq n_j.
	\ee
We use \eqref{eq_sum}, substitute \eqref{eq_bdnj} in \eqref{eq_upfin} and invoke \eqref{eq_disc} to get
	\bea
		v(\Om;\polcl{f}{n})&\leq& (1-\eps)^{-21}\ \lkor \left(\sum\limits_{j=1}^m
		\lam(q_j)^\frac{1}{3}s(L_j)\right)^\frac{3}{2}\frac{1}{\sqrt{n}}\notag \\
		&\leq &(1-\eps)^{-24}\: \frac{\lkor}{2}  \left(\int_{\bdy\Om}\sigma_\Om\right)^\frac{3}{2}\frac{1}{\sqrt{n}}\label{eq_FINUP},
	\eea
for $n$ sufficiently large.

%----------
%---- LOWER BOUND ----
%----------
$\bm{2.}$ Next, we produce a lower bound for $v(\Om;\polcl{f}{n})$. For this, we first extend the tiling $\{L_j\}_{1\leq j\leq m}$ of $\bdy\Om$ to a thin tubular neighborhood of $\bdy\Om$ in $\overline\Om$, denoting the tile corresponding to $L_j$ by $\hat L_j$. This can be done, for instance, by flowing each tile along the inward normal vector field for a short interval of time. Choose a $P_n\in\polcl{f}{n}$ such that $\vol(\Om\setminus P_n)\leq (1-\eps)^{-1}v(\Om;\polcl{f}{n})$. Let $n_j$ be the number of cuts of $P_n$ that cover $J_j$. Due to the upper bound obtained in \eqref{eq_FINUP}, $\lim_{n\rightarrow\infty}v(\Om;\polcl{f}{n})=0$. Thus, by Lemma \ref{lem_shrunbdd}, $\lim_{n\rightarrow\infty}\de(P_n)=0$. This permits us to choose $n$ sufficiently large so that 
	\begin{itemize}
%\item The complement in $\Om$ of the $n_j$ cuts that cover $J_j$ is in $\polcl{J_j\subset L_j 	 	 ;f}{n_j}$. 
\item The $n_j$ cuts that cover $J_j$ lie in $\widehat L_j$. 
\item Each $n_j$ is large enough so that the bounds in Lemma \ref{lem_locfin} hold.
\end{itemize}
Thus, invoking Lemma \ref{lem_locfin} and using \eqref{eq_JKest}, we have that 
	\beas 
		\vol(\Om\setminus P_n)\geq \sum\limits_{j=1}^m \vol\left(\widehat L_j\setminus P_n\right) 
		&\geq&  \sum\limits_{j=1}^m  v(\Om;\polcl{J_j\subset L_j;f}{n_j})\\
		& \geq&  (1-\eps)^{31}\ \lkor\sum\limits_{j=1}^m
		 \frac{\lam(q_j)^{\frac{1}{2}}
			 s(J_j)^{\frac{3}{2}}}{\sqrt{n_j}}\\
		& \geq&(1-\eps)^{33}\ \lkor\sum\limits_{j=1}^m
		 \frac{\lam(q_j)^{\frac{1}{2}}
			 s(L_j)^{\frac{3}{2}}}{\sqrt{n_j}}.
	\eeas
Now, H{\"o}lder's inequality gives
	\beas
 		 \sum\limits_{j=1}^m \lam(q_j)^\frac{1}{3}s(L_j)
		= \sum\limits_{j=1}^m \left(\frac{\lam(q_j)s(L_j)^3}{n_j}\right)^\frac{1}{3}n_j^{\frac{1}{3}}
			\leq \left(\sum\limits_{j=1}^m \frac{\lam(q_j)^\frac{1}{2}s(L_j)^\frac{3}{2}}{\sqrt{n_j}}\right)^\frac{2}{3} 	
					\left(\sum\limits_{j=1}^m n_j\right)^\frac{1}{3}
	\eeas
Thus, using one of the estimates in \eqref{eq_disc},
	\beas		
		\vol(\Om\setminus P_n)&\geq&(1-\eps)^{33}\ \lkor
		\left(\sum\limits_{j=1}^m \lam(q_j)^\frac{1}{3}s(L_j)\right)^\frac{3}{2}\frac{1}{\sqrt{n_1+\cdots+n_m}}\\
		&\geq & (1-\eps)^{35}\ \lkor
		\left(\frac{1}{4^{1/3}}\int_{\bdy\Om}\sigma_\Om\right)^\frac{3}{2}\frac{1}{\sqrt{n}}.
	\eeas
By our choice of $P_n$,
	\be\label{eq_FINLO}
		v(\Om;\polcl{f}{n}) \geq 	(1-\eps)^{36}\:\frac{\lkor}{2}
		\left(\int_{\bdy\Om}\sigma_\Om\right)^\frac{3}{2}\frac{1}{\sqrt{n}},
	\ee
for all $n$ sufficiently large. 

Finally, we combine \eqref{eq_FINLO} and \eqref{eq_FINUP}, and recall that $\eps\in(0,1/3)$ was arbitrary, to declare the proof of Theorem \ref{thm_MAIN} complete. 
\qed

%%%%%%%%%%%%%%%%%%%%%%%%% 
%%%%%%%%%%%%%%%%%%%%%%%%% 
%%%%%%%%%%%%%%%%%%%%%%%%% Appendix %%%%%%%%%%%%%%%%%%%%%%%%%%%%
%%%%%%%%%%%%%%%%%%%%%%%%% 
%%%%%%%%%%%%%%%%%%%%%%%%% 

\section{Appendix: Power Diagrams in the Heisenberg Group} 

\noindent {\bf The Euclidean Plane.} Let $D(a;r)\subset \rl^2$ be a disk of radius $r$ centered at $a\in\rl^2$. The {\em power} of a point $z=(x,y)\in \rl^2$ with respect to $D=D(a;r)$ is the number
	\bes
		\pow(z,D)=|z-a|^2-r^2.
	\ees
Note that if $z$ is outside the disk $D$, then $\pow(z,D)$ is the square of the length of a line segment from $z$ to a point of tangency with $\bdy D$. Thus, it is a generalized distance between $z$ and $\bdy D$. For a collection $\mathscr{D}$ of disks in the plane, the {\em power diagram} or {\em Laguerre-Dirichlet-Voronoi tiling} of $\mathscr{D}$ is the collection of all 
	\bes
		\cell(D)=\{z\in\rl^2:\pow(z,D)< \pow(z,D^*), \forall D^*\in \mathscr{D}\setminus\{D\}\},
		\ D\in\mathscr{D}.
	\ees
If $\mathscr{D}$ consists of equiradial disks, the power diagram reduces to the Dirichlet-Voronoi diagram of the centers of the disks. In general, the power diagram of any $\mathscr{D}$ gives a convex tiling of the plane. 
\begin{figure}[H]
\centering
\resizebox{2.2in}{!}{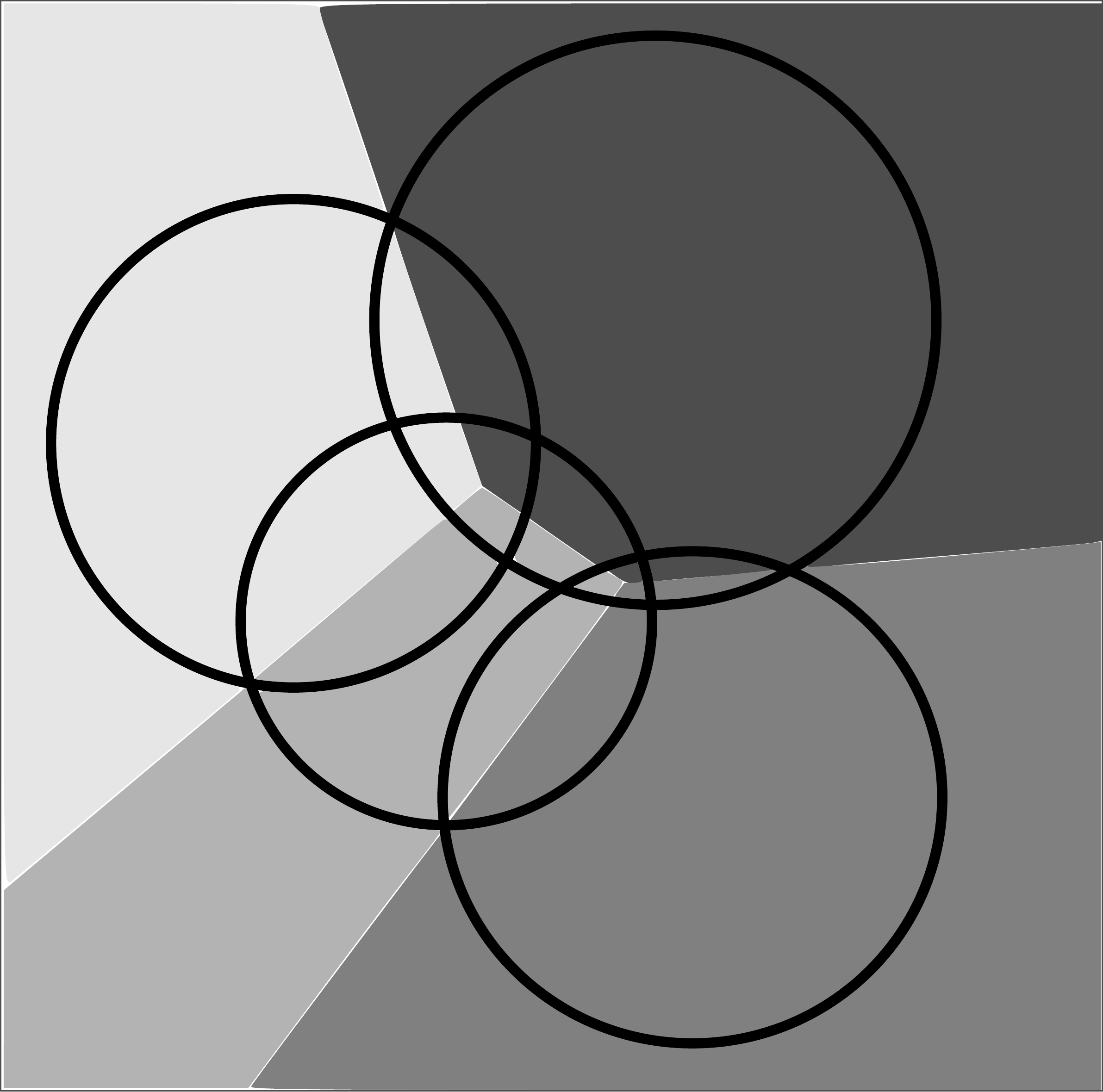}
\caption{A power diagram in the plane.}
\end{figure} 
Power diagrams occur naturally and have found several applications (see \cite{Aur}, for instance). From the point of view of polyhedral approximations, power diagrams (in $\rl^{d-1}$) are intimately related to the constant $\ldiv_{d-1}$ in \eqref{eq_lud} (see \cite{Lu} and \cite{BoLu} for explicit details). 

\noindent {\bf The Heisenberg Group.} Let $K(0;\de)=\{z'\in\bbh:|z_1|^4+(x_2)^2<\de^4\}$ be a Kor{\'a}nyi sphere in $\bbh$ (see \eqref{eq_proj_cuts}). We define the {\em horizontal power} of a point $z'\in\bbh$ with respect to $K=K(0;\de)$ as 
	\bes
		\hpow(z',K)=
		\begin{cases}
			|z_1|^2-\sqrt{\de^4-(x_2)^2}, & \text{if}\ |x_2|\leq \de^2;\\
			 \infty, & \text{otherwise}.
		\end{cases}
	\ees
Note that $K_c:=K\cap \{x_2=c\} $ is a (possibly empty) disk in the $\{x_2=c\}$ plane, and $\hpow((z_1,x_2),K)=\pow(z_1,K_{x_2})$, where the right-hand side --- being a generalized distance --- is set as $\infty$ when $K_{x_2}$ is empty. $\hpow$ is then extended to all Kor{\'a}nyi spheres to be left-invariant under $\cdot_{\scriptscriptstyle\bbh}$ (defined in Section \ref{sec_approxmodel}). For a collection $\mathscr{K}$ of Kor{\'a}nyi spheres in $\bbh$, define the {\em horizontal power diagram} or {\em Laguerre-Kor{\'a}nyi tiling} of $\mathscr{K}$ to be the collection of all 
	\bes
	\hcell(K)=\left\{z'\in \bigcup_{K\in\mathscr{K}}K:\hpow(z',K)<\hpow(z',K^*),\forall K^*\in\mathscr{K}\setminus\{K\}\right\},\ K\in\mathscr{K}.
	\ees
Then, $\hcell(K)\subset$ $K$, for all $K\in\mathscr{K}$. 
\begin{figure}[H]
\centering
\resizebox{2.2in}{!}{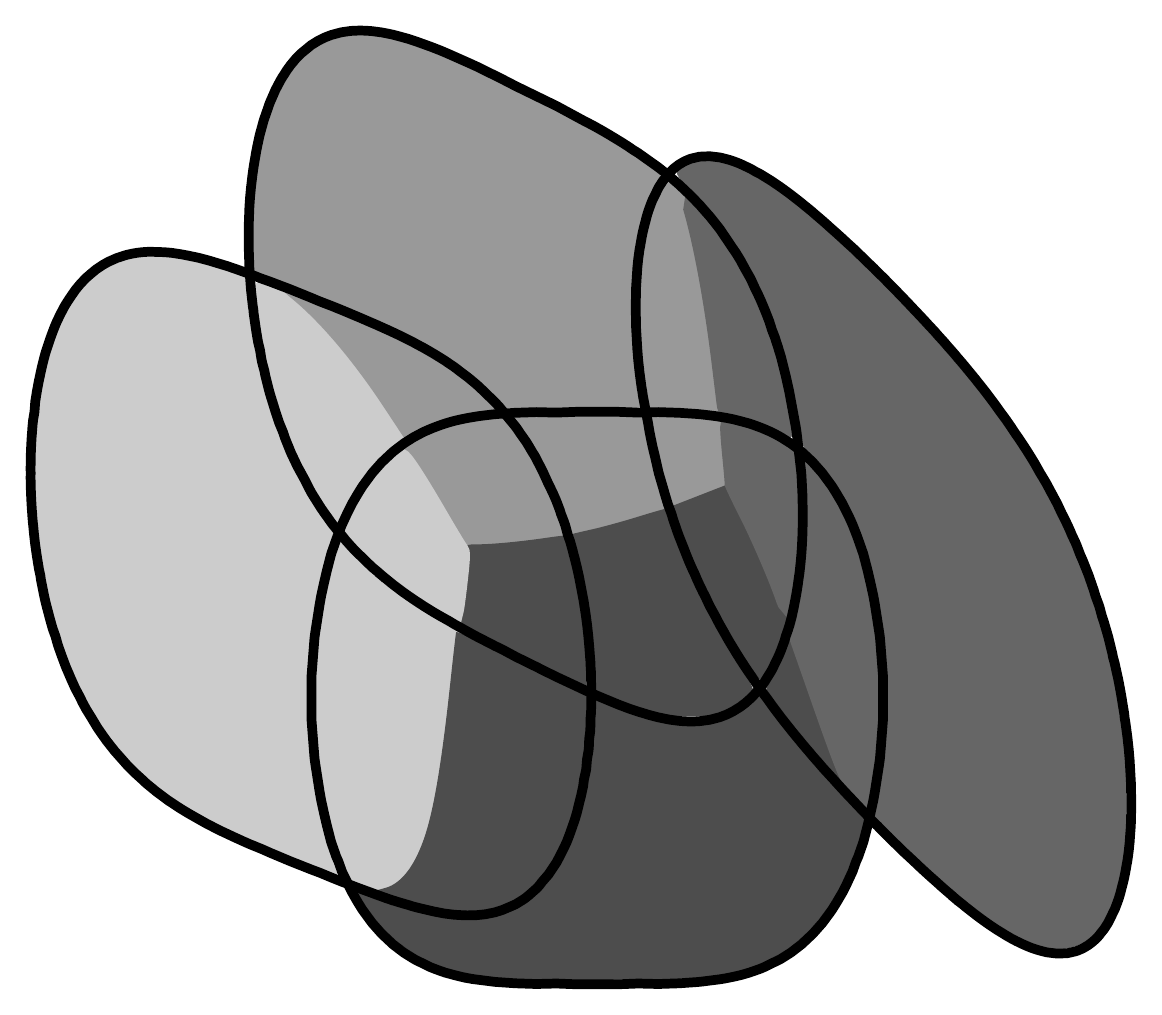}
\caption{A $\{x_1=0\}$-slice of a horizontal power diagram in $\bbh$.}
\end{figure} 

We now give two reasons why this concept is useful for us. Let 
	\beas
		&&\dil_\xi:(z_1,x_2)\mapsto (\xi z_1,\xi^2 x_2),\\
		&&\dil_{w',\xi}:z'\mapsto w'\doth\dil_\xi(-w'\doth z')
	\eeas
be the dilations in $\bbh$ centered at the origin and $w'$, respectively. Then,
\begin{enumerate}
\item $\dil_{w',\xi}(K(w',\de))=K(w',\xi\de)$,
\item $\hpow(\dil_{w',\xi}(z'),K(w',\de))=\xi^2\hpow(z',K(w',\xi^{-1}\de))$, and
\item if $\mathscr{K}=\{K_j:=K(a_j,\de_j):j=1,...,m\}$, then, $\dil_{a_j,\xi}\hcell\big(K_l\big)
		\cap\dil_{a_k,\xi}\hcell\big(K_j\big)=\emptyset$, for all $1\leq l<j\leq m$ and $\xi\leq 1$.
\end{enumerate}

Now, consider the Siegel domain $\mathcal{S}$ and the function $f_\mathcal{S}$ studied in Section \ref{sec_approxmodel}. The cuts of any $f_\mathcal{S}$-polyhedron $P$ over $J\subset\bdy S$ project to a collection $\mathscr{K}_P$ of Kor{\'a}nyi balls in $\C\times\rl$ that form a covering of $J'$. The (open) facets of $P$ project to the horizontal power diagram of $\mathscr{K}_P$. This perspective facilitates the proof of
%----------
%----Admissible Cuts ----
%----------
\begin{lemma}\label{lem_adm}
The cuts of $f_{\mathcal{S}_\lam}$, $\lam>0$, are Jordan measurable and satisfy the doubling property \eqref{eq_doub} for any $\de_{f_{\mathcal{S}_\lam}}>0$ and $D(t)=(1+t)^3$. 
\end{lemma}
\begin{proof} The Jordan measurability of the cuts is obvious. Now, without loss of generality, we may assume $\lam=1$ (the map $(z,w)\mapsto (\lam z,\lam w)$ can be used to handle the other cases). Let $H\subset\bdy \mathcal{S}$ be a compact set, $\{w^j\}_{1\leq j \leq m}\subset H$, $\{\de_j\}_{1\leq j \leq m}\subset (0,\infty)$ and $t>0$. For $j=1,...,m$, let 
	\beas
		&C_j(t):=C(w_j,(1+t)\de_j;f_\mathcal{S}),&\\
		& v^j=(w^j)'=(w_1^j,u_2^j),&
	\eeas
and (see \eqref{eq_proj_cuts})
	\bes
		K_j(t):=C_j(t)'=K\left(v^j;\sqrt{(1+t)\de_j}\right).
	\ees
Consider $\mathscr{K}=\{K_j(t):1\leq j\leq m\}$ and the corresponding horizontal power diagram $\{\hcell_j(t)=\hcell(K_j(t)):1\leq j\leq m\}$. Then, setting $dz'=dx_1dy_1dx_2$, we have, by a change of variables and $(1)$, $(2)$ and $(3)$ above, that
\beas
		&&\vol\left(\bigcup\limits_{j=1}^{m}C_j(t)\right)\\
		&=&\int\limits_{\cup_{j=1}^m K_j(t)}
			\max\limits_{1\leq j\leq m}\left\{\rea{\sqrt{\de_j^2-(x_2-u_2^j+2\ima z_1\conj{w_1}^j)}}-|z_1-w_1^j|^2\right\}dz'\\
		&=&\int\limits_{\cup_{j=1}^m K_j(t)}\max\limits_{1\leq j\leq m}\{-\hpow(z',K_j(t))\}dz'\\
		&=&
		-\sum\limits_{j=1}^{m}\int\limits_{\hcell_j(t)}\hpow( z',K_j(t))
			dz'\\
		&=&-(1+t)^2\sum\limits_{j=1}^{m}\int_{\dil_{v^j,\frac{1}{\sqrt{1+t}}}(\hcell_j(t))}
			\hpow\left(\dil_{v^j,\sqrt{1+t}}(\zeta),K_j(t)\right)d\zeta			\\
		&=&-(1+t)^3\sum\limits_{j=1}^{m}\int_{\dil_{v^j,\frac{1}{\sqrt{1+t}}}(\hcell_j(t))}
			\hpow\left(\zeta,K_j(0)\right)d\zeta		\\
		&\leq& (1+t)^3\int_{\cup_{j=1}^m K_j(0)}
			\max\left\{-\hpow\left(\zeta,K_j(0)\right):1\leq j\leq m\right\}d\zeta\\
		&=& (1+t)^3\vol\left(\bigcup\limits_{j=1}^{m}C_j(0)\right),\ \forall t\geq 0.
	\eeas
\end{proof}

The computations in the above proof also show that
\bes
	\lkor=\lim_{n\rightarrow\infty}\sqrt{n}\inf\left\{-\sum\limits_{K\in\mathscr{K}}\int_{\hcell(K)}\hpow(z',K)dz':I\subset \bigcup_{K\in\mathscr{K}}K,\ \#(\mathscr{K})\leq n \right\},
\ees
where $I$ is the unit square in $\C\times\rl$ (see Section \ref{sec_approxmodel}). Our proof of Lemma \ref{lem_heis} yields bounds for $\lkor$ as follows:
\bes
	0.0003\approx\frac{4\sqrt{2}}{\pi^23^7}\leq \lkor\leq \frac{5\sqrt{5}\pi}{3\sqrt{2}}\approx 8.2788.
\ees
It would be interesting to know if computations, similar to the ones carried out by B{\"o}r{\"o}czky and Ludwig in \cite{BoLu} for $\ldiv_2$, can be done to find the exact value of $\lkor$.  

%%%%%%%%%%%%%%%%%%%%%%%%% 
%%%%%%%%%%%%%%%%%%%%%%%%% 
%%%%%%%%%%%%%%%%%%%%%%%%% Appendix %%%%%%%%%%%%%%%%%%%%%%%%%%%%
%%%%%%%%%%%%%%%%%%%%%%%%% 
%%%%%%%%%%%%%%%%%%%%%%%%% 
%\nocite{*}
\bibliography{Gupta_bib}
\bibliographystyle{plain}
\end{document}